\documentclass[11pt,reqno]{amsart}
\usepackage{enumitem}
\usepackage{amsfonts,amssymb,amsthm}
\usepackage{amsmath}
\usepackage[margin=1.3in]{geometry}
\usepackage[
linktocpage=true,colorlinks,citecolor=magenta,linkcolor=blue,urlcolor=magenta]{hyperref}
\usepackage{amsrefs}
\newtheorem{proposition}{Proposition}[section]
\newtheorem{theorem}[proposition]{Theorem}
\newtheorem{lemma}[proposition]{Lemma}
\newtheorem{corollary}[proposition]{Corollary}
\theoremstyle{definition}
\newtheorem{definition}[proposition]{Definition}
\theoremstyle{remark}

\numberwithin{equation}{section}
\allowdisplaybreaks

\newcommand{\Hh}{\mathbb H}
\newcommand{\Sph}{\mathbb S}
\newcommand{\Ric}{\operatorname{Ric}}
\newcommand{\tr}{\operatorname{tr}}
\newcommand{\Id}{\operatorname{Id}}
\newcommand{\eps}{\varepsilon}
\newcommand{\CC}{\mathcal C}

\author[T. Luo]{Tianci Luo}
\author[Y. Wei]{Yong Wei}
\author[R. Zhou]{Rong Zhou}
\address{School of Mathematical Sciences, University of Science and Technology of China, Hefei 230026, P.R. China}
\email{Luo\_tianci@mail.ustc.edu.cn}
\email{yongwei@ustc.edu.cn}
\email{zhourong@mail.ustc.edu.cn}

\date{\today}

\begin{document}

\title[\texorpdfstring{$\tau$}{tau}-bi-Ricci curvature and MCF surgery]{Positive \texorpdfstring{$\tau$}{tau}-bi-Ricci curvature and Mean curvature flow with surgery in hyperbolic space}

\subjclass[2020]{53E10, 53C42, 53C21}
\keywords{mean curvature flow, hyperbolic space, bi-Ricci curvature,  surgery, cylindrical estimate}

\begin{abstract}
Let $n\geq3$ and $0\leq\tau\leq2$.  We prove that every smooth closed
connected immersed hypersurface in hyperbolic space whose induced metric has
positive $\tau$-bi-Ricci curvature admits a mean curvature flow with surgery
which has only finitely many surgery times and terminates.
In the range compatible with cylindrical necks, the key ingredients are a
preserved quantitative spectral pinching condition, which converts the
intrinsic hypothesis into uniform two-convexity, and cylindrical and
derivative estimates that remain valid across the hyperbolic standard neck
replacement.  At the endpoint $n=3$, $\tau=2$, positive $\tau$-bi-Ricci
curvature is positive Ricci curvature and forces strict convexity, so the
ordinary mean curvature flow converges to a round point.  Consequently, the
underlying manifold is diffeomorphic to a sphere or to a finite connected sum
of copies of $\Sph^{n-1}\times\Sph^1$.  If the initial hypersurface is embedded
and bounds a compact domain, that domain is a one-handlebody, namely a ball
with finitely many one-handles attached.
\end{abstract}

\maketitle
\setcounter{tocdepth}{1}
\tableofcontents

\section{Introduction}

Mean curvature flow provides a way to convert local curvature hypotheses into
global information about the topology of a hypersurface.  Huisken
\cites{Hui84,Hui86} proved round-point convergence for closed convex
hypersurfaces in Euclidean space and for suitably convex hypersurfaces in
Riemannian ambient manifolds.  For nonconvex two-convex hypersurfaces in
Euclidean space, Huisken and Sinestrari
\cites{HS99sing,HS99conv,HS09} developed the convexity, cylindrical and
derivative estimates that make it possible to identify high-curvature necks
and replace them by standard caps.  Andrews, Langford and McCoy \cite{ALM14}
and Andrews and Langford \cite{AL14} extended the convexity and cylindrical
estimates to broad classes of fully nonlinear convex curvature flows.
Haslhofer and Kleiner \cite{HK17mean} developed a local structure theory for
mean-convex and $k$-convex mean curvature flows.  In the surgery setting,
Brendle and Huisken \cite{BH16} treated closed embedded mean-convex surfaces
in $\mathbb R^3$, and Haslhofer and Kleiner \cite{HK17} gave a local
construction based on noncollapsing estimates.  Brendle and Huisken
\cite{BH17} extended the surgery picture to two-convex hypersurfaces in
general Riemannian ambient manifolds by using a fully nonlinear flow adapted
to the ambient curvature. More recently, Sz\'ekelyhidi
\cite{Szekelyhidi26} constructed a mean curvature flow with surgery starting
from any compact mean-convex hypersurface in Euclidean space.  In this
construction the topological changes occur through nondegenerate cylindrical
singularities.

Related estimates have also been obtained in compact space forms and
symmetric spaces.  Nguyen \cite{Nguyen15} proved convexity and cylindrical
estimates for mean curvature flow in the round sphere, and Pipoli and
Sinestrari \cite{PS17} established cylindrical estimates in compact rank-one
symmetric spaces.  Langford and Nguyen \cite{LN21} constructed a terminating
mean curvature flow with surgery for quadratically pinched hypersurfaces of
the sphere.  More recently, they \cite{LN25} established
asymptotically sharp quadratic pinching, cylindrical, and derivative
estimates for mean curvature flow in the sphere, with applications to
convexity estimates and singularity models.  Their pinching condition is
extrinsic and is related to, but distinct from, two-convexity.  In hyperbolic
space, Andrews and Chen
\cite{AC17}*{Theorem~2} proved round-point convergence for embedded
hypersurfaces of positive intrinsic Ricci curvature, while Ji \cite{Ji21}
obtained a cylindrical estimate under an extrinsic pinching hypothesis.

The result of Andrews and Chen \cite{AC17} is the closest smooth-flow analogue
of the present paper.  In dimension three, positive $2$-bi-Ricci curvature is
exactly positive Ricci curvature.  At this endpoint the condition forces local
strict convexity.  The hyperbolic Hadamard--Stoker theorem \cite{Alexander77}  then shows that the initial hypersurface is embedded, so the smooth round-point theorem
of Andrews and Chen applies.  Away from this endpoint, the $\tau$-bi-Ricci
conditions are compatible with two-convex but nonconvex hypersurfaces and with
cylindrical high-curvature regions.  The present work treats immersed initial
data and seeks a topological classification in the regime where neck
singularities can occur.

\subsection{Main result}

This paper asks whether an intrinsic curvature inequality can supply all of
the structure required by surgery and remain valid when a neck is replaced by
caps.  Let $X_0:M^n\to\Hh^{n+1}$ be a smooth closed connected immersion,
where $n\geq3$ and $\Hh^{n+1}$ has sectional curvature $-1$.  For orthonormal tangent
vectors $u,v\in TM$, set
\begin{equation}\label{eq:intro-taubiric}
 \operatorname{BiRic}_{\tau}(u,v)
 =\Ric(u,u)+\Ric(v,v)-\tau\operatorname{Sec}(u\wedge v).
\end{equation}
At $\tau=1$, this is the bi-Ricci curvature introduced by Shen and Ye
\cite{SY96} in their study of stable minimal surfaces.  It is also the $m=2$
intermediate curvature of Brendle, Hirsch and Johne \cite{BHJ24}.  At
$\tau=0$, positivity is equivalent to positivity of the sum of the two
smallest eigenvalues of the Ricci tensor, namely the two-positive Ricci
condition studied in \cite{CW22,CM26,Wolf}.  The parameter $\tau$
therefore connects two intrinsic curvature conditions with different
topological origins.  For hypersurfaces, the Gauss equation turns this family
into a condition on the second fundamental form and reveals exactly
which part of it is compatible with the cylindrical neck model.

If $\lambda_1\leq\cdots\leq\lambda_n$ are the principal curvatures, the
Gauss equation expresses the intrinsic condition in terms of the second
fundamental form.  Positivity of \eqref{eq:intro-taubiric} is equivalent to
\begin{equation*}
 P_{ij}^{(\tau)}
 =H(\lambda_i+\lambda_j)-\lambda_i^2-\lambda_j^2
  -\tau\lambda_i\lambda_j-\bigl(2(n-1)-\tau\bigr)>0
 \qquad(i\ne j).
\end{equation*}
Summing these inequalities gives positive scalar curvature.  The scalar Gauss
equation then implies that the mean-curvature vector never vanishes.  Hence the
normal bundle is trivial, and we choose the global normal so that $H>n$.  With
this orientation, we evolve the immersion by
\begin{equation}\label{eq:mcf}
\begin{cases}
    \partial_tX=-H\nu,\\
    X(\cdot,0)=X_0.
\end{cases}
\end{equation}
The condition is intrinsic, but the formula for $P_{ij}^{(\tau)}$ supplies the
extrinsic control needed by the flow.  Its quantitative form implies uniform
two-convexity, while it does not require pointwise convexity.  Cylindrical
singularities therefore remain possible and must be treated by surgery.

There are two restrictions relevant to the surgery construction.  The
invariant-region argument for the smooth flow uses the concavity of the
spectral inequalities defining $P_{ij}^{(\tau)}$ and requires
$0\leq\tau\leq2$.  The cylindrical neck model imposes a second restriction.
Let $\mathcal T_\rho$ be the tube of radius $\rho$ about a geodesic in
$\Hh^{n+1}$.  With the mean-convex orientation its principal curvatures are
$(\tanh\rho,\coth\rho,\ldots,\coth\rho)$.  For two spherical principal
directions, the Gauss expression gives
\begin{equation*}
 P_{ij}^{(\tau)}
 =\bigl(2(n-2)-\tau\bigr)\operatorname{csch}^2\rho.
\end{equation*}
Thus the tubular cylinder has strictly positive $\tau$-bi-Ricci curvature
exactly when $\tau<2(n-2)$.  As $\rho\downarrow0$, these tubes are the
hyperbolic models whose blow-ups converge to the round Euclidean cylinder.
We call
\begin{equation*}
 \mathcal I_n=
 \begin{cases}
 [0,2),&n=3,\\
 [0,2],&n\geq4
 \end{cases}
\end{equation*}
the nonconvex surgery range.  Equivalently, $\mathcal I_n=[0,2]\cap\{\tau:\tau<2(n-2)\}$.  
The bound $\tau\leq2$ is required by the smooth-flow invariant region, while
the strict inequality is imposed by the neck model.  At $n=3$, $\tau=2$,
the cylindrical tube has $P_{23}^{(2)}=0$ for every radius.  Hence a standard
neck replacement cannot preserve strict positive $2$-bi-Ricci curvature.
This endpoint is nevertheless included in the main theorem through a
separate smooth-flow argument.

The main theorem gives both the flow and the resulting topological
classification.

\begin{theorem}
\label{thm:taubiric-main}
Let $n\geq3$ and $0\leq\tau\leq2$.  Let
$X_0:M^n\to\Hh^{n+1}$ be a smooth closed connected immersion whose induced
metric has positive $\tau$-bi-Ricci curvature.  Then there is a mean curvature
flow with surgery starting from $X_0$ which has only finitely many surgery
times and terminates.  When $n=3$ and $\tau=2$, the flow
may be taken to be smooth and it converges to a round point.  Moreover,
\begin{equation*}
 M\cong\Sph^n
 \quad\text{or}\quad
 M\cong\mathop{\#}_{j=1}^{k}
       \bigl(\Sph^{n-1}\times\Sph^1\bigr)
\end{equation*}
for some $k\geq1$.  If $X_0$ is embedded and bounds a compact domain
$\Omega_0$, the surgery flow may be chosen embedded and $\Omega_0$ is a
one-handlebody, namely a ball with finitely many one-handles attached.
\end{theorem}

The topological alternatives in Theorem~\ref{thm:taubiric-main} are familiar
from two-convex surgery.  The new point is the intrinsic curvature
hypothesis under which they are obtained.  Positive bi-Ricci curvature,
and more generally positive intermediate curvature, allows substantial
topological flexibility on abstract Riemannian manifolds.  For instance,
positive bi-Ricci curvature is preserved under connected sums
\cite{ShenYe97}, while two-positive Ricci curvature occurs on broad
classes of highly connected manifolds \cite{CW22} and, in
general, yields only weaker restrictions on the fundamental group
\cite{Wolf}.  See also the recent work \cite{MWY26} on the topology
of manifolds with positive intermediate curvature.  For hypersurfaces
in hyperbolic space, however, the Gauss equation converts the intrinsic
$\tau$-bi-Ricci condition into a preserved quantitative spectral
pinching condition.  It is this additional hypersurface structure that
leads to the diffeomorphism classification above.

\subsection{Outline of the proof}

For $\tau\in\mathcal I_n$, the proof follows the surgery construction of
Huisken and Sinestrari \cite{HS09} and its space-form adaptation by Langford
and Nguyen \cite{LN21}.  It proceeds in four steps.  The surgery argument
differs from the preceding theories in three places.  We preserve a
quantitative form of the original intrinsic curvature condition, prove the
high-curvature estimates uniformly as the ambient hyperbolic curvature
disappears under rescaling, and verify the same intrinsic condition directly
on the surgically inserted cap.  The endpoint $n=3$, $\tau=2$ is treated separately by the smooth-flow argument in Section~\ref{sec:topology}.

First, Section~\ref{sec:preliminaries} converts the intrinsic hypothesis into
a preserved spectral condition.  Positive $\tau$-bi-Ricci curvature makes
the mean-curvature vector nonzero and therefore determines a global
mean-convex orientation with $H>n$.  Compactness
strengthens the strict inequalities to
\begin{equation*}
 P_{ij}^{(\tau)}
 \geq\alpha_0H|\lambda_p-\lambda_q|
 \qquad(i\ne j,\ p,q\text{ arbitrary})
\end{equation*}
for some $\alpha_0>0$.  The associated set of Weingarten maps is closed,
convex and $O(n)$-invariant.  Hamilton's vector-bundle maximum principle
\cite{Ham86}, in the form used by Andrews and Hopper \cite{AH11}*{Section 7.4} and Andrews
and Chen \cite{AC17}*{Proposition 4}, preserves this set along the smooth flow.  In
particular,
\begin{equation*}
 \lambda_1+\lambda_2\geq\beta H
\end{equation*}
for some $\beta>0$.  The stronger inequality involving $P_{ij}^{(\tau)}$ is
retained rather than discarded after two-convexity is obtained.  It is used
again in both the cylindrical estimate and the surgery calculation.  This
also distinguishes the present condition from the extrinsic quadratic
pinching in \cite{LN21}, which is invariant under reversal of the normal.

Second, Sections~\ref{sec:cylindrical} and
\ref{sec:gradient-estimates} establish the estimates needed to recognize
high-curvature necks.  The central point is a cylindrical estimate whose
constants are uniform for ambient sectional curvature $-\kappa^2$ with
$0\leq\kappa\leq1$.  On the region where
\begin{equation*}
 |A|^2-\frac1{n-1}H^2\geq\eta H^2,
\end{equation*}
the quantitative $\tau$-bi-Ricci inequality rules out every algebraic zero of
the normalized Simons commutator.  Compactness then gives the coercive bound
$|\mathfrak S|^2\geq\gamma H^6$.  Combining this bound with the space-form
Simons identity gives a Poincar\'e-type inequality on the acylindrical region.
The weighted evolution inequality and Stampacchia iteration yield
\begin{equation*}
 |A|^2-\frac1{n-1}H^2
 \leq\eta H^2+C_\eta R^{-2}.
\end{equation*}
The evolution equations for the gradient and Hessian of $A$ then give the
higher derivative bounds.  A blow-up at a point with large $H$ has rescaled
ambient curvature tending to zero.  The uniform estimates pass to this
Euclidean limit, where the strong maximum principle identifies the limiting
flow as a round shrinking cylinder.  This proves the neck-detection theorem.

Third, Section~\ref{sec:standard-surgery} shows that the quantitative
$\tau$-bi-Ricci inequality survives the standard neck replacement.  On the
 normalized cylinder, the two possible leading margins in
 $P_{ij}^{(\tau)}$ are $n-2$ and $2(n-2)-\tau$.  Within the analytically
 allowed range $0\leq\tau\leq2$, their positivity gives exactly
 $\tau\in\mathcal I_n$.  The standard cap has a positive leading margin for
 $0\leq\tau\leq2$, and the same remains true along the bending and rotational
 interpolation.  At surgery scale $r$, the errors from hyperbolic
normal coordinates are of lower order after rescaling.  By choosing the
surgery curvature sufficiently large, we absorb these errors and reduce the
pinching constant only once.  The resulting constant is independent of the
number of later surgeries.  This direct intrinsic verification is additional
to the standard preservation of two-convexity.

Finally, the cylindrical and derivative estimates are extended across
surgery times, and the neck-continuation and neck-selection arguments are
carried out as in \cites{HS09,LN21}, with the hyperbolic replacement estimates
compared with \cite{BH17}.  Each surgery at the fixed surgery scale decreases
area by a definite amount, while the positive lower bound for $H-n$ gives a
uniform upper bound for the total elapsed time.  Hence the process terminates
after finitely many replacements.  Reversing the neck cuts gives the
connected-sum classification in Theorem~\ref{thm:taubiric-main}.  In the
embedded case, the same reconstruction shows that the enclosed domain is a
one-handlebody. 

\subsection{Organization of the paper}

Section~\ref{sec:preliminaries} develops the invariant $\tau$-bi-Ricci region
and its elementary consequences.  Section~\ref{sec:cylindrical} proves the
cylindrical estimate, and Section~\ref{sec:gradient-estimates} proves the
derivative estimates and neck detection for smooth flows.
Section~\ref{sec:standard-surgery} checks the standard replacement and
extends the estimates to surgically modified flows.
Section~\ref{sec:parameter-choice} constructs the terminating flow, and
Section~\ref{sec:topology} treats the smooth endpoint separately, deduces the
topology, and completes the proof of Theorem~\ref{thm:taubiric-main}.

\section{Preservation of the \texorpdfstring{$\tau$}{tau}-bi-Ricci condition}
\label{sec:preliminaries}

Throughout Sections~2--4 the ambient manifold is the simply connected space form $\Hh^{n+1}_{\kappa}$ of sectional curvature $-\kappa^2$, where $0\leq\kappa\leq1$. The value $\kappa=0$ denotes Euclidean space. This section fixes the hypersurface notation in this uniform normalization. It also relates the intrinsic $\tau$-bi-Ricci condition to the principal curvatures, proves its preservation under the smooth flow, and records the bounds used later in the surgery construction.

\subsection{Hypersurfaces and mean curvature flow}

For $\kappa>0$, geodesic polar coordinates about a fixed point identify
$\Hh^{n+1}_{\kappa}\setminus\{o\}$ with
$\Sph^n\times(0,\infty)$ and give the metric
\begin{equation*}
 \bar g_{\kappa}=dr^2+\kappa^{-2}\sinh^2(\kappa r)g_{\Sph^n}.
\end{equation*}
The limit $\kappa=0$ gives the Euclidean polar metric
$dr^2+r^2g_{\Sph^n}$.  
Let $X:M^n\to\Hh^{n+1}_{\kappa}$ be a smooth two-sided immersion and choose a
unit normal $\nu$.  We denote the Levi-Civita connections of $\bar g_\kappa$
and the induced metric $g$ by $\bar\nabla$ and $\nabla$, respectively.  In
local coordinates on $M$, with $\partial_i=\partial X/\partial x^i$, set
\begin{equation*}
 g_{ij}=\bar g_\kappa(\partial_i,\partial_j),
 \qquad
 h_{ij}=\bar g_\kappa(\bar\nabla_{\partial_i}\nu,\partial_j).
\end{equation*}
Thus $h$ is the second fundamental form and
$A=(h_i{}^j)$, where $h_i{}^j=g^{jk}h_{ik}$, is the Weingarten map.  Our sign
convention makes geodesic spheres positively curved.  The eigenvalues of $A$
are the principal curvatures, ordered as
\begin{equation*}
 \lambda_1\leq\cdots\leq\lambda_n.
\end{equation*}
We write
\begin{equation*}
 H=\tr A=\sum_i\lambda_i,
 \qquad
 |A|^2=\tr(A^2)=\sum_i\lambda_i^2,
 \qquad
 (h^2)_{ij}=h_i{}^kh_{kj}.
\end{equation*}
The Gauss and Codazzi equations take the form
\begin{align}
 R_{ijkl}
 &=h_{ik}h_{jl}-h_{il}h_{jk}
   -\kappa^2(g_{ik}g_{jl}-g_{il}g_{jk}),
   \label{eq:gauss-space-form}\\
 \nabla_i h_{jk}&=\nabla_jh_{ik}.\notag
\end{align}

Now let $X:M\times[0,T)\longrightarrow\Hh^{n+1}_{\kappa}$
solve the mean curvature flow \eqref{eq:mcf}, and write
$M_t=X(M,t)$.  We use the standard variation formulas for hypersurfaces
moving by mean curvature.   See \cite{Ecker04,HS09,Hui86,AC17,LN21}.
\begin{align}
 \partial_tg_{ij}=&~-2Hh_{ij},\nonumber\\
 \partial_t d\mu=&~-H^2d\mu, \label{s2.evl-dmu}\\
    \partial_th_{ij}
 =&~\nabla_i\nabla_jH-H\bigl((h^2)_{ij}+\kappa^2g_{ij}\bigr)\nonumber\\
 \partial_th_{ij}  =&~\Delta h_{ij}-2H\bigl((h^2)_{ij}+\kappa^2g_{ij}\bigr)
  +(|A|^2+n\kappa^2)h_{ij},\label{eq:h-direct-evolution}\\
 (\partial_t-\Delta)H=&~(|A|^2-n\kappa^2)H,\label{eq:H-evol}\\
 (\partial_t-\Delta)H^2=&-2|\nabla H|^2
   +2H^2(|A|^2-n\kappa^2),\label{eq:H2-evol}\\
 (\partial_t-\Delta)|A|^2=&-2|\nabla A|^2
   +2|A|^2(|A|^2-n\kappa^2)
   +4n\kappa^2\left(|A|^2-\frac1nH^2\right)\label{eq:A2-evol}\\
 (\partial_t-\Delta)|\nabla A|^2
 \leq&-2|\nabla^2A|^2+c(|A|^2+\kappa^2)|\nabla A|^2,\label{eq:derivative-evolutions}\\
 (\partial_t-\Delta)|\nabla^2A|^2
 \leq&-2|\nabla^3A|^2
 +c\bigl((|A|^2+\kappa^2)|\nabla^2A|^2+|\nabla A|^4\bigr), \label{eq:derivative-evolutions-2}
\end{align}
where $c=c(n)$.

For the preservation argument we use the time-dependent connection introduced in this context by Andrews--Baker \cite{AB10}, see also
\cite{LN21}*{Section~2.2}. Regard $h(t)$ as a section of
$\operatorname{Sym}^2(T^*M)$ and equip the time-dependent tangent bundle with the temporal connection
\begin{equation*}
 \nabla_t\partial_i=-Hh_i{}^j\partial_j.
\end{equation*}
The metric evolution gives $\nabla_t g=0$. With respect to this connection, the second fundamental form satisfies
\begin{equation}\label{eq:nabla-t-h}
 \nabla_th_{ij}=\Delta h_{ij}+(|A|^2+n\kappa^2)h_{ij}
 -2\kappa^2Hg_{ij}.
\end{equation}
The reaction field in this equation is
\begin{equation*}
 \mathcal Q_{\kappa}(A)=(|A|^2+n\kappa^2)A
 -2\kappa^2(\tr A)\Id.
\end{equation*}
In a principal frame its eigenvalue vector is
\begin{equation*}
 Q_{\kappa,i}=\lambda_i(|A|^2+n\kappa^2)-2\kappa^2H.
\end{equation*}
For a differentiable function $\varphi$ of the principal curvatures, we
write
\begin{equation*}
 Q_\kappa\varphi=\sum_{i=1}^nQ_{\kappa,i}
 \frac{\partial\varphi}{\partial\lambda_i}.
\end{equation*}
This is the form of the equation to which Hamilton's vector-bundle maximum principle \cite{Ham86} applies. The same formulation is used in related preservation arguments for curvature flows. Compare
\cite{HS09}*{Proposition~2.6} and \cite{AC17}*{Proposition~4}.

\subsection{The invariant \texorpdfstring{$\tau$}{tau}-bi-Ricci pinching region}

Fix $0\leq\tau\leq2$. For orthonormal tangent vectors $u,v$, let
$\operatorname{BiRic}_{\tau}(u,v)$ be the curvature quantity defined in \eqref{eq:intro-taubiric}.  In an orthonormal principal frame, the Gauss equation \eqref{eq:gauss-space-form} gives
\begin{align}
K_{ij}=&~\lambda_i\lambda_j-\kappa^2 ,\label{s2.Kij}\\
\rho_i=&~\lambda_i(H-\lambda_i)-(n-1)\kappa^2, \label{s2.rhoi}
\end{align}
and we put
\begin{equation}\label{eq:Ptau-definition}
 P_{ij}^{(\tau)}
 :=\rho_i+\rho_j-\tau K_{ij}.
\end{equation}

\begin{lemma}[Reduction to principal directions]
The induced metric has nonnegative, respectively positive,
$\tau$-bi-Ricci curvature if and only if
\begin{equation*}
 P_{ij}^{(\tau)}\geq0,
 \qquad\text{respectively }P_{ij}^{(\tau)}>0,
\end{equation*}
for every $i\ne j$ in a principal frame.
\end{lemma}

\begin{proof}
Let $u=\sum_i u_i e_i$ and $v=\sum_i v_i e_i$ be orthonormal, and set
$\omega_{ij}=u_i v_j-u_j v_i$. Since the curvature operator of a hypersurface in a space form is diagonal in the basis $e_i\wedge e_j$,
\begin{equation*}
 \operatorname{Sec}(u\wedge v)
 =\sum_{i<j}K_{ij}\omega_{ij}^2.
\end{equation*}
Moreover,
\begin{equation*}
 \sum_{j\ne i}\omega_{ij}^2=u_i^2+v_i^2.
\end{equation*}
Consequently,
\begin{align*}
 \operatorname{BiRic}_{\tau}(u,v)
 &=\sum_i\rho_i(u_i^2+v_i^2)
   -\tau\sum_{i<j}K_{ij}\omega_{ij}^2\\
 &=\sum_{i<j}
   \bigl(\rho_i+\rho_j-\tau K_{ij}\bigr)\omega_{ij}^2
 =\sum_{i<j}P_{ij}^{(\tau)}\omega_{ij}^2.
\end{align*}
Because $\sum_{i<j}\omega_{ij}^2=|u\wedge v|^2=1$, positivity of all
principal coefficients is sufficient.  Necessity follows by taking
$u=e_i$ and $v=e_j$.
\end{proof}

The principal expression is
\begin{equation}\label{eq:Ptau-expanded}
 P_{ij}^{(\tau)}
 =H(\lambda_i+\lambda_j)-\lambda_i^2-\lambda_j^2
  -\tau\lambda_i\lambda_j-c_{n,\tau}\kappa^2,
\end{equation}
where $c_{n,\tau}:=2(n-1)-\tau$.

\begin{definition}
After choosing the normal so that $H\geq n\kappa$, a hypersurface satisfies the
\emph{quantitative $\tau$-bi-Ricci pinching condition} with constant
$\alpha_0>0$ if
\begin{equation}\label{eq:uniform-taubiric}
 P_{ij}^{(\tau)}
 \geq\alpha_0 H|\lambda_p-\lambda_q|
 \qquad(i\ne j,\ p,q\text{ arbitrary}).
\end{equation}
\end{definition}

\begin{lemma}
\label{lem:taubiric-mean-convex}
If the induced metric has nonnegative $\tau$-bi-Ricci curvature, then its
scalar curvature is nonnegative and the norm of its mean-curvature vector is
at least $n\kappa$.  If $\kappa>0$, the mean-curvature vector is nowhere zero.
Thus the normal can be chosen so
that $H\geq n\kappa$.  Under strict positivity, the same conclusion holds for
every $\kappa\geq0$, with $H>n\kappa$.
\end{lemma}

\begin{proof}
Since the scalar curvature $\mathrm{Scal}=\sum_i\rho_i=2\sum_{i<j}K_{ij}$,
\begin{equation*}
 \sum_{i<j}P_{ij}^{(\tau)}
 =\left(n-1-\frac\tau2\right)\mathrm{Scal}.
\end{equation*}
The coefficient is positive for $n\geq3$ and $0\leq\tau\leq2$.  Thus
nonnegative $\tau$-bi-Ricci curvature implies $\mathrm{Scal}\geq0$, and
strict positivity implies $\mathrm{Scal}>0$.  The Gauss equation
\begin{equation*}
 \mathrm{Scal}=H^2-|A|^2-n(n-1)\kappa^2
\end{equation*}
and $|A|^2\geq H^2/n$ give $H^2\geq n^2\kappa^2$, with strict inequality in
the positive case.  Here $H$ is computed using any local unit normal, while
$H^2$ is the squared norm of the globally defined mean-curvature vector.  If
$\kappa>0$, or if the curvature condition is strict, this vector is nowhere
zero and trivializes the normal line bundle.  Choosing the resulting global
normal so that $H$ is positive gives the stated inequalities.  In the
Euclidean nonnegative case, an orientation with $H\geq0$ must instead be
imposed as a separate hypothesis.
\end{proof}

\begin{theorem}[Preservation of quantitative $\tau$-bi-Ricci pinching]
\label{thm:taubiric-preservation}
Let $0\leq\kappa\leq1$ and $0\leq\tau\leq2$, and choose the normal so
that $H\geq n\kappa$.  Nonnegative $\tau$-bi-Ricci curvature is preserved by
smooth mean curvature flow.  More generally, for every fixed
$\alpha_0>0$, the quantitative pinching condition
\eqref{eq:uniform-taubiric} is preserved whenever it holds initially.
Every compact hypersurface with positive $\tau$-bi-Ricci curvature satisfies
\eqref{eq:uniform-taubiric} for some $\alpha_0>0$.
\end{theorem}

\begin{proof}
Fix $\alpha_0\geq0$.  We first construct a closed convex set whose defining
inequalities are exactly the quantitative pinching condition.  The division by
$H$ below is useful because it makes the defining functions concave.  On the
half-space $H=\sum_k\lambda_k>0$, define
\begin{equation*}
 G_{ijpq}^{(\tau)}(\lambda)
 =\lambda_i+\lambda_j
 -\frac{\lambda_i^2+\lambda_j^2+\tau\lambda_i\lambda_j
              +c_{n,\tau}\kappa^2}{H}
 -\alpha_0(\lambda_p-\lambda_q).
\end{equation*}
Then
\begin{equation}\label{s2.HG}
 HG_{ijpq}^{(\tau)}
 =P_{ij}^{(\tau)}-\alpha_0 H(\lambda_p-\lambda_q).
\end{equation}
Because the ordered pairs $(p,q)$ and $(q,p)$ are both included, the
inequalities $G_{ijpq}^{(\tau)}\geq0$ are exactly
\eqref{eq:uniform-taubiric}.  Let $\Omega_{\tau,\alpha_0}$ be the closure in $\mathbb{R}^n$ of the set 
\begin{equation*}
\left\{\lambda\in\mathbb R^n:H>0,\ H\geq n\kappa,\quad
 G_{ijpq}^{(\tau)}(\lambda)\geq0
 \text{ for all }i\ne j,p,q\right\}.
\end{equation*}
When $\kappa>0$, the closure is redundant. When $\kappa=0$, the inequality  $G_{ijpq}^{(\tau)}(\lambda)\geq0$ implies $P_{ij}^{(\tau)}\geq0$ and hence $H^2-|A|^2\geq0$. Therefore the only additional point in the closure with $H=0$  is the origin.

For $\xi\in\mathbb R^n$, write
$\bar\xi=\sum_k\xi_k$ and
\begin{equation*}
 a_i=\xi_i-\frac{\lambda_i}{H}\bar\xi,
 \qquad a_j=\xi_j-\frac{\lambda_j}{H}\bar\xi.
\end{equation*}
Direct differentiation yields
\begin{equation*}
 D^2G_{ijpq}^{(\tau)}[\xi,\xi]
 =-\frac2H\bigl(a_i^2+a_j^2+\tau a_i a_j\bigr)
   -\frac{2c_{n,\tau}\kappa^2}{H^3}\bar\xi^2\leq0.
\end{equation*}
The quadratic form $a_i^2+a_j^2+\tau a_i a_j$ has eigenvalues
$1\pm\tau/2$, and $c_{n,\tau}\kappa^2\geq0$.  Hence every
$G_{ijpq}^{(\tau)}$ is concave for
 $0\leq\tau\leq2$. The set inside the preceding closure is therefore convex, and so $\Omega_{\tau,\alpha_0}$ is a closed, convex,
 permutation-invariant subset of $\mathbb R^n$.  By the spectral convexity theorem of Lewis \cite{Lewis96}*{Corollary~2.4}, applied to the indicator function of $\Omega_{\tau,\alpha_0}$, the corresponding set
\begin{equation*}
 \mathcal K_{\tau,\alpha_0}
 =\{A\in\operatorname{Sym}^2(\mathbb R^n):
      \lambda(A)\in\Omega_{\tau,\alpha_0}\}
\end{equation*}
is closed, convex, and $O(n)$-invariant.

Hamilton's \cite{Ham86} vector-bundle maximum principle reduces preservation of this set to an inward-pointing condition for the reaction equation obtained from
\eqref{eq:nabla-t-h}.  Since
$\mathcal K_{\tau,\alpha_0}$ is closed, convex, and $O(n)$-invariant, it is
enough to show that $\mathcal Q_\kappa(A)$ lies in its tangent cone at every
boundary point.  The boundary has the faces defined by active inequalities
$G_{ijpq}^{(\tau)}=0$ and the face $H=n\kappa$.  We check these in turn.

We begin with an active $G_{ijpq}^{(\tau)}$-face.  To compute the derivative of
$P_{ij}^{(\tau)}$ along the reaction equation, set
\begin{align*}
 \mathcal S_i
 :=&\sum_{r,s\ne i}(\lambda_r-\lambda_s)^2,\\
 E_{ij}:=&H(\lambda_i+\lambda_j)-|A|^2-n\lambda_i\lambda_j,
\end{align*}
where the sum in $\mathcal S_i$ is over ordered pairs.  We give the
reaction calculation explicitly.  First,
\begin{equation}\label{s2.QkH}
 Q_\kappa H
 =\sum_{r=1}^n Q_{\kappa,r}
 =H\bigl(|A|^2-n\kappa^2\bigr).
\end{equation}
Using \eqref{s2.rhoi},  the product rule gives
\begin{align*}
 Q_\kappa\rho_i
 &= (H-2\lambda_i)Q_{\kappa,i}
    +\lambda_iQ_\kappa H\\
 &=2|A|^2\rho_i
   +\kappa^2\left(
      2(n-1)|A|^2-2n\lambda_i^2
      -2H^2+4H\lambda_i
   \right).
\end{align*}
The expression in parentheses is precisely $\mathcal S_i$.  Indeed,
because the sum in the definition of $\mathcal S_i$ is over ordered
pairs,
\begin{align*}
 \mathcal S_i
 &=2(n-1)\sum_{r\ne i}\lambda_r^2
   -2\left(\sum_{r\ne i}\lambda_r\right)^2\\
 &=2(n-1)|A|^2-2n\lambda_i^2
   -2H^2+4H\lambda_i.
\end{align*}
Consequently,
\begin{equation}\label{s2.Qrho}
 Q_\kappa\rho_i
 =2|A|^2\rho_i+\kappa^2\mathcal S_i.
\end{equation}
Compare also the corresponding calculation in \cite{AC17}.

For the sectional-curvature term \eqref{s2.Kij}, we have
\begin{align}\label{s2.QK}
 Q_\kappa(\lambda_i\lambda_j-\kappa^2)
 &=\lambda_jQ_{\kappa,i}
   +\lambda_iQ_{\kappa,j}\nonumber\\
 &=2(|A|^2+n\kappa^2)\lambda_i\lambda_j
   -2\kappa^2H(\lambda_i+\lambda_j).
\end{align}
Combining \eqref{s2.Qrho} and \eqref{s2.QK}, and using \eqref{eq:Ptau-definition} 
and the definition of $E_{ij}$, we therefore obtain
\begin{align}
 Q_\kappa P_{ij}^{(\tau)}
 &=2|A|^2P_{ij}^{(\tau)}
   +\kappa^2\bigl(
      \mathcal S_i+\mathcal S_j+2\tau E_{ij}
   \bigr).
 \label{s2.QP}
\end{align}
The last bracket in \eqref{s2.QP} is nonnegative.  In fact, the identity
\begin{align*}
 \mathcal S_i+\mathcal S_j+4E_{ij}
 ={}&2(n-2)(\lambda_i-\lambda_j)^2\\
 &+4\left[(n-2)\sum_{k\ne i,j}\lambda_k^2
           -\left(\sum_{k\ne i,j}\lambda_k\right)^2\right]\geq0
\end{align*}
follows from Cauchy's inequality.  Consequently, for $0\leq\tau\leq2$,
\begin{align}
 \mathcal S_i+\mathcal S_j+2\tau E_{ij}
 ={}&\left(1-\frac\tau2\right)(\mathcal S_i+\mathcal S_j)\notag\\
 &+\frac\tau2(\mathcal S_i+\mathcal S_j+4E_{ij})\geq0.
 \label{eq:tau-reaction-convex-combination}
\end{align}
The coefficients in this convex combination are nonnegative precisely for
$0\leq\tau\leq2$.  This is the sign needed for the reaction vector to point
into an active $G_{ijpq}^{(\tau)}$-face.

We also derive the reaction of the quantitative pinching term.  Since
the scalar term $-2\kappa^2H$ in $Q_{\kappa,i}$ is independent of $i$,
it cancels in the difference and gives
\begin{equation*}
 Q_\kappa(\lambda_p-\lambda_q)
 =(|A|^2+n\kappa^2)(\lambda_p-\lambda_q).
\end{equation*}
Together with \eqref{s2.QkH}, the product rule yields
\begin{align}
 Q_\kappa\bigl(H(\lambda_p-\lambda_q)\bigr)
 &=Q_\kappa H\,(\lambda_p-\lambda_q)
   +H\,Q_\kappa(\lambda_p-\lambda_q)\notag\\
 &=2|A|^2H(\lambda_p-\lambda_q).
 \label{s2.QH}
\end{align}
Therefore, using \eqref{s2.QP}, \eqref{s2.QH}, and \eqref{s2.HG},
\begin{align}\label{s2.QHG}
 Q_\kappa\left(H G_{ijpq}^{(\tau)}\right)=&Q_\kappa \left(P_{ij}^{(\tau)}-\alpha_0 H(\lambda_p-\lambda_q)\right)\nonumber\\
 =&2|A|^2\left(H G_{ijpq}^{(\tau)}\right)
   +\kappa^2\bigl(\mathcal S_i+\mathcal S_j+2\tau E_{ij}\bigr).
\end{align}
At an active supporting face $G_{ijpq}^{(\tau)}=0$ of
$\mathcal K_{\tau,\alpha_0}$, the product rule gives
\begin{equation*}
 Q_\kappa\left(H G_{ijpq}^{(\tau)}\right)
 =G_{ijpq}^{(\tau)}Q_\kappa H
   +H\,Q_\kappa G_{ijpq}^{(\tau)}
 =H\,Q_\kappa G_{ijpq}^{(\tau)}.
\end{equation*}
On the other hand, \eqref{s2.QHG} and
\eqref{eq:tau-reaction-convex-combination} give
\begin{equation*}
 Q_\kappa\left(H G_{ijpq}^{(\tau)}\right)
 =\kappa^2
   \bigl(\mathcal S_i+\mathcal S_j+2\tau E_{ij}\bigr)
 \geq0
\end{equation*}
on the same face.  Since $H>0$ there, we conclude that
\begin{equation*}
 Q_\kappa G_{ijpq}^{(\tau)}\geq0.
\end{equation*}
Thus the reaction field points into the supporting half-space determined by
every active constraint.

It remains to check the boundary arising from the mean-curvature
condition.  Suppose first that $\kappa>0$ and $H=n\kappa$.   
Using $|A|^2\geq H^2/n$, from \eqref{s2.QkH} we have
\begin{equation*}
 Q_\kappa(H-n\kappa)
 =H\bigl(|A|^2-n\kappa^2\bigr)
 \geq0.
\end{equation*}
Thus the reaction field points into the half-space $H\geq n\kappa$.
When $\kappa=0$, the preceding description of
$\Omega_{\tau,\alpha_0}$ shows that the only point added by taking the
closure with $H=0$ is the origin.  At the corresponding matrix $A=0$,
\begin{equation*}
 Q_0(A)=|A|^2A=0,
\end{equation*}
so the closure creates no additional boundary issue.

We also explain why the preceding calculation is unaffected by
multiple eigenvalues.  The reaction map has the form
\begin{equation*}
 Q_\kappa(A)
 =\bigl(|A|^2+n\kappa^2\bigr)A
  -2\kappa^2H\operatorname{Id}.
\end{equation*}
Hence $Q_\kappa(A)$ commutes with $A$ and restricts to a scalar
endomorphism on every eigenspace of $A$.  Equivalently,
\begin{equation*}
 \lambda_i=\lambda_j
 \quad\Longrightarrow\quad
 Q_{\kappa,i}=Q_{\kappa,j}.
\end{equation*}
Thus the induced reaction equation for the eigenvalues is well defined
independently of the choice of orthonormal basis inside a multiple
eigenspace.  Since $\Omega_{\tau,\alpha_0}$ is permutation invariant,
the active-constraint calculation above therefore applies without
ambiguity at repeated eigenvalues.

We have consequently shown that the reaction equation preserves
$\mathcal K_{\tau,\alpha_0}$.  Since this set is closed, convex, and
$O(n)$-invariant, it defines a parallel family of fiberwise convex sets
for the metric temporal connection.  Hamilton's vector-bundle maximum
principle \cite{Ham86}, applied to \eqref{eq:nabla-t-h}, therefore shows
that $\mathcal K_{\tau,\alpha_0}$ is preserved by the full parabolic
equation.  Compare the corresponding invariant-region arguments in
\cite{HS09}*{Proposition~2.6} and \cite{AC17}*{Proposition~4}.
Taking $\alpha_0=0$ gives preservation of nonnegative
$\tau$-bi-Ricci curvature, while every fixed $\alpha_0>0$ gives
preservation of the quantitative pinching condition
\eqref{eq:uniform-taubiric}.

Finally, suppose that the initial hypersurface is compact and has
positive $\tau$-bi-Ricci curvature.  By
Lemma~\ref{lem:taubiric-mean-convex}, choose the normal so that
$H>n\kappa$.  Compactness gives
\begin{equation*}
 \min_{M_0}\min_{i\ne j}P_{ij}^{(\tau)}>0.
\end{equation*}
If
\begin{equation*}
 \max_{M_0}H(\lambda_n-\lambda_1)>0,
\end{equation*}
choose $\alpha_0>0$ sufficiently small that
\begin{equation*}
 \alpha_0\max_{M_0}H(\lambda_n-\lambda_1)
 \leq
 \min_{M_0}\min_{i\ne j}P_{ij}^{(\tau)}.
\end{equation*}
Since $ |\lambda_p-\lambda_q|
 \leq\lambda_n-\lambda_1$,  
the quantitative pinching condition
\eqref{eq:uniform-taubiric} follows.  If the maximum above vanishes,
then all principal curvatures agree at every point and the right-hand
side of \eqref{eq:uniform-taubiric} is identically zero, so any
$\alpha_0>0$ may be chosen.  The preservation already proved then
applies for the whole smooth flow.
\end{proof}

\subsection{Surgery classes and elementary consequences}
The surgery class used below is modeled on the class
$\mathcal C(R,\alpha)$ introduced by Huisken--Sinestrari
\cite{HS09}*{Definition~2.5} and on the surgery class used by
Langford--Nguyen \cite{LN21}*{Section~3.4}.  We replace their curvature
hypotheses by the quantitative $\tau$-bi-Ricci pinching established above
and retain only the parameters needed in the present hyperbolic setting.

\begin{definition}[Surgery class]
Fix $0\leq\kappa\leq1$, $\tau\in\mathcal I_n$, $0<R\leq1$, and a triple
$\alpha=(\alpha_0,\alpha_1,\alpha_2)$ of positive constants.  We denote by
$\CC_{\tau,\kappa}(R,\alpha)$ the class of smooth closed two-sided
hypersurface immersions $X:M^n\longrightarrow\Hh^{n+1}_{\kappa}$ satisfying the following properties:
\begin{enumerate}[label=\textup{(\roman*)},leftmargin=*]
\item $M$ satisfies the quantitative $\tau$-bi-Ricci pinching
\eqref{eq:uniform-taubiric} with constant $\alpha_0$.
\item The mean curvature has the uniform positive margin
\begin{equation}\label{eq:tau-surgery-class-bounds}
 H-n\kappa\geq\alpha_2R^{-1}.
\end{equation}
\item The induced area satisfies $ |M|\leq\alpha_1R^n$. 
\end{enumerate}
\end{definition}

The initial curvature normalization $|A|^2\leq R^{-2}$ is kept separate from the definition of the surgery class.  This is analogous to the normalization used in the Euclidean surgery construction of Huisken--Sinestrari \cite{HS09}. It is needed here for the lower lifespan bound below, whereas a cap inserted at surgery is controlled at the surgery scale rather than at the initial scale $R$.

For a fixed initial hypersurface with $H>n\kappa$, choose $R$ so that
\begin{equation*}
 R^{-1}\geq\max\{1,\|A\|_{L^\infty(M_0)}\},
\end{equation*}
and then choose
\begin{equation*}
 0<\alpha_2\leq R\min_{M_0}(H-n\kappa).
\end{equation*}
Both the curvature normalization and
\eqref{eq:tau-surgery-class-bounds} then hold.

The class is preserved along every smooth time interval.  Condition
\textup{(i)} is preserved by
Theorem~\ref{thm:taubiric-preservation}.  At a spatial minimum of $H$,
\eqref{eq:H-evol} and $|A|^2\geq H^2/n$ give
\begin{equation*}
 (\partial_t-\Delta)H
 \geq \frac{H}{n}(H^2-n^2\kappa^2)\geq0,
\end{equation*}
because $H\geq n\kappa$.  Hence the lower bound in \textup{(ii)} cannot
decrease.  Finally, \eqref{s2.evl-dmu}  
shows that the area is non-increasing, and therefore preserves
\textup{(iii)}.

\begin{proposition}[Uniform two-convexity]
\label{prop:uniform-two-convexity}
There is $\beta=\beta(n,\alpha_0)>0$ such that every hypersurface in
$\CC_{\tau,\kappa}(R,\alpha)$ satisfies
\begin{equation}\label{eq:uniform-two-convexity}
 \lambda_1+\lambda_2\geq\beta H.
\end{equation}
One may take
\begin{equation}\label{eq:beta-tau}
 \beta=\frac{2\alpha_0}{2(n-2)+n\alpha_0}.
\end{equation}
Consequently,
\begin{equation*}
 -H\leq\lambda_1\leq\cdots\leq\lambda_n\leq H,
 \qquad \frac1nH^2\leq|A|^2\leq nH^2,
\end{equation*}
$\lambda_i\geq\beta H/2$ for $i\geq2$, and
\begin{equation}\label{eq:kato-hs}
 |H\nabla_i h_{kl}-(\nabla_iH)h_{kl}|^2
 \geq\frac{\beta^2}{8}H^2|\nabla H|^2.
\end{equation}
\end{proposition}

\begin{proof}
The point requiring proof here is the quantitative two-convexity constant
deduced from the intrinsic $\tau$-bi-Ricci pinching.  Once this estimate is
known, all remaining conclusions are the standard consequences recorded in
\cite{HS09}*{Proposition~2.7}.

Since $0\leq\tau\leq2$,
\begin{equation*}
 \lambda_1^2+\lambda_2^2+\tau\lambda_1\lambda_2\geq0,
 \qquad c_{n,\tau}\kappa^2\geq0.
\end{equation*}
Hence \eqref{eq:Ptau-expanded} gives
\begin{equation*}
 H(\lambda_1+\lambda_2)\geq P_{12}^{(\tau)}
 \geq\alpha_0H(\lambda_n-\lambda_1),
\end{equation*}
and therefore
\begin{equation}\label{eq:s-controls-width-tau}
 \lambda_1+\lambda_2\geq\alpha_0(\lambda_n-\lambda_1).
\end{equation}
The ordering gives
\begin{equation*}
 H\leq \lambda_1+\lambda_2+(n-2)\lambda_n,
 \qquad \lambda_1\leq\frac {\lambda_1+\lambda_2}{2},
\end{equation*}
so
\begin{equation}\label{s2.lambda_n-1}
 \lambda_n-\lambda_1
 \geq\frac{2H-n(\lambda_1+\lambda_2)}{2(n-2)}.
\end{equation}
If the right-hand side of \eqref{s2.lambda_n-1} is nonpositive, then
$\lambda_1+\lambda_2\geq2H/n\geq\beta H$, since the value in
\eqref{eq:beta-tau} satisfies $\beta\leq2/n$.  Otherwise,
combining \eqref{s2.lambda_n-1} with \eqref{eq:s-controls-width-tau} yields
\begin{equation*}
 2(n-2)(\lambda_1+\lambda_2)\geq\alpha_0(2H-n(\lambda_1+\lambda_2)),
\end{equation*}
which gives \eqref{eq:uniform-two-convexity} and
\eqref{eq:beta-tau}.  Applying
\cite{HS09}*{Proposition~2.7(i)--(iii)} with two-convexity constant $\beta$
gives the principal-curvature bounds, the estimate for $|A|^2$, and
\eqref{eq:kato-hs}.  That argument is algebraic apart from the Codazzi
identity and therefore applies unchanged in a space form.
\end{proof}

\begin{lemma}\label{lem:time-upper}
Let $M_0\in\CC_{\tau,\kappa}(R,\alpha)$, assume that
$|A|^2\leq R^{-2}$ on $M_0$, and let $[0,T)$ be the maximal interval on
which the mean curvature flow starting from $M_0$ is smooth.  Then
\begin{equation}\label{eq:time-upper}
 2\tau_nR^2\leq T\leq\frac n2\alpha_2^{-2}R^2,
\end{equation}
where we denote $ \tau_n=\frac{1}{4n}\log(1+n)$. 
\end{lemma}

\begin{proof}
Let $m(t)=\min_{M_t}H$ and set
$y(t)=m(t)-n\kappa$.  At a spatial minimum, \eqref{eq:H-evol} and
$|A|^2\geq H^2/n$ give, in the barrier sense,
\begin{equation*}
 m'\geq\frac1n m(m^2-n^2\kappa^2).
\end{equation*}
Since $m=y+n\kappa$, we have 
\begin{equation*}
 m\bigl(m^2-n^2\kappa^2\bigr)
 =(y+n\kappa)y(y+2n\kappa)
 \geq y^3.
\end{equation*}
Hence 
\begin{equation*}
 y'\geq\frac1n y^3,
 \qquad y(0)\geq\alpha_2R^{-1}.
\end{equation*}
The solution of $z'=z^3/n$ with
$z(0)=\alpha_2R^{-1}$ becomes unbounded at
$n\alpha_2^{-2}R^2/2$.  Scalar comparison gives $y(t)\geq z(t)$ as long as both are defined, and  so
\begin{equation*}
 T\leq\frac n2\alpha_2^{-2}R^2.
\end{equation*}

We next prove the lower bound.  From \eqref{eq:A2-evol},
\begin{align*}
 (\partial_t-\Delta)|A|^2
 &\leq
 2|A|^4+2n\kappa^2|A|^2-4\kappa^2H^2\\
 &\leq
 2|A|^4+2n|A|^2,
\end{align*}
where $0\leq\kappa\leq1$ was used in the last inequality.  Thus $u(t):=\max_{M_t}|A|^2$  
is bounded from above, in the barrier sense, by the solution of
\begin{equation*}
 v'=2v^2+2nv,
 \qquad
 v(0)=R^{-2}.
\end{equation*}
The blow-up time of this comparison solution is
\begin{equation*}
 t_*
 =
 \int_{R^{-2}}^\infty
 \frac{ds}{2s(s+n)}
 =
 \frac1{2n}\log(1+nR^2).
\end{equation*}
In particular, $|A|$ remains bounded on every compact subinterval of
$[0,t_*)$.  If $T<t_*$, the same comparison gives a uniform bound for
$|A|$ on $[0,T)$, and the standard continuation criterion for compact
mean curvature flow extends the solution past $T$.  This contradiction
shows that $ T\geq t_*$.  

Finally, since $0<R\leq1$ and the function
\begin{equation*}
 s\longmapsto\log(1+ns)
\end{equation*}
is concave on $[0,1]$ and vanishes at $s=0$, we have
\begin{equation*}
 \log(1+nR^2)
 \geq
 R^2\log(1+n).
\end{equation*}
Therefore
\begin{equation*}
 T \geq
 \frac{\log(1+n)}{2n}R^2
 =  2\tau_nR^2.
\end{equation*}
This proves \eqref{eq:time-upper}.
\end{proof}

\section{The cylindrical estimate}\label{sec:cylindrical}

In this section we establish the cylindrical estimate for the smooth mean
curvature flow.  The overall strategy follows 
Huisken--Sinestrari \cite{HS09}*{Section~5} and its space-form version due to
Langford--Nguyen \cite{LN21}*{Section~4.1}.  The new ingredient in the present
setting is an algebraic coercivity estimate for the full Simons commutator
which follows from the quantitative $\tau$-bi-Ricci pinching established in
Section~\ref{sec:preliminaries}.  This yields a Poincar\'e-type inequality on
the region where the hypersurface is quantitatively away from being
cylindrical.  The remaining argument proceeds through the standard
$L^p$ estimates and Stampacchia iteration.


\begin{theorem}[Cylindrical estimate]\label{thm:cylindrical}
Let $M_t$, $t\in[0,T)$, be a smooth solution of \eqref{eq:mcf} with
initial hypersurface
$M_0\in\CC_{\tau,\kappa}(R,\alpha)$, and assume in addition that
$|A|^2\leq R^{-2}$ on $M_0$.  Then, for every $\eta>0$, there exists
$C_\eta=C_\eta(n,\alpha,\eta)>0$, independent of $\kappa$, $R$, and $T$,
such that
\begin{equation}\label{eq:cylindrical}
 |A|^2-\frac{1}{n-1}H^2
 \leq \eta H^2+C_\eta R^{-2}
\end{equation}
on $M_t$ for every $t\in[0,T)$.
\end{theorem}

\subsection{Uniform space-form normalization}\label{subsec:space-form-normalization}

We first reduce the proof to a fixed curvature scale.  Let $X:M\times[0,T)\longrightarrow\Hh^{n+1}_{\kappa}$  
be the given mean curvature flow.  Define the rescaled ambient metric and
time variable by
\begin{equation*}
 \widehat{\bar g}:=R^{-2}\bar g,
 \qquad
 \widehat t:=R^{-2}t,
\end{equation*}
and set $ \widehat X(x,\widehat t):=X(x,R^2\widehat t)$.  
The Levi-Civita connection is unchanged under this constant rescaling, while
the unit normal and the curvature quantities transform according to
\begin{equation*}
 \widehat\nu=R\nu,\qquad
 \widehat\lambda_i=R\lambda_i,\qquad
 \widehat H=RH,\qquad
 |\widehat A|=R|A|.
\end{equation*}
Consequently,
\begin{equation*}
 \frac{\partial\widehat X}{\partial\widehat t}
 =R^2\frac{\partial X}{\partial t}
 =-R^2H\nu
 =-\widehat H\,\widehat\nu,
\end{equation*}
so $\widehat X$ is again a mean curvature flow in a space form of curvature $-\widehat\kappa^2$, where $ \widehat\kappa=R\kappa$.   Its maximal smooth time is $\widehat T=R^{-2}T$.  

Moreover,
\begin{equation*}
 |\widehat A_0|^2=R^2|A_0|^2\leq1,
 \qquad
 |\widehat M_0|=R^{-n}|M_0|\leq\alpha_1,
\end{equation*}
and $\widehat H-n\widehat\kappa
 =R(H-n\kappa)\geq\alpha_2$.  
The quantitative $\tau$-bi-Ricci condition is invariant under this
simultaneous rescaling of the principal curvatures and the ambient curvature
parameter.  Since $0<R\leq1$ and $0\leq\kappa\leq1$, we also have
$0\leq\widehat\kappa\leq1$.

It therefore suffices to prove all estimates under the normalized assumptions
\begin{equation}\label{eq:normalized-space-form-setting}
 R=1,\qquad
 0\leq\kappa\leq1,\qquad
 H\geq n\kappa+\alpha_2,\qquad
 |A_0|^2\leq1,\qquad
 |M_0|\leq\alpha_1.
\end{equation}
In this normalization Lemma~\ref{lem:time-upper} also gives the uniform
upper bound
\begin{equation}\label{eq:normalized-time-upper}
 T\leq\frac n2\alpha_2^{-2}.
\end{equation}
In the remainder of the proof we work under
\eqref{eq:normalized-space-form-setting}, omit the hats, and restore the
original scale only at the end.

\subsection{The Simons commutator and a Poincar\'e inequality}
The key ingredient in the integral argument is a Poincar\'e-type inequality
on the region where the curvature is quantitatively away from the cylindrical
configuration.  We follow the general strategy of
 Langford--Nguyen \cite{LN21}*{Proposition~2.2} (see also \cite{LN25}*{Proposition~3.5}), which is based on the
 Simons identity and an algebraic lower bound for its commutator.
 The algebraic step is different in the present setting.  Instead of the
 quadratic curvature pinching used in \cite{LN21}, the required coercivity
 will follow from the quantitative $\tau$-bi-Ricci pinching
 \eqref{eq:uniform-taubiric}.

For a hypersurface in a space form, the four-index Simons identity takes the form
\begin{equation}\label{eq:full-simons}
 \nabla_{(i}\nabla_{j)}h_{kl}
 -\nabla_{(k}\nabla_{l)}h_{ij}
 =\mathfrak S_{ijkl}.
\end{equation}
Here parentheses denote normalized symmetrization. For example,
\begin{equation*}
 \nabla_{(i}\nabla_{j)}h_{kl}
 =\frac12\left(
 \nabla_i\nabla_jh_{kl}+\nabla_j\nabla_ih_{kl}
 \right).
\end{equation*} 
In our normalization, where the ambient sectional
curvature is $-\kappa^2$, the algebraic commutator is
\begin{equation}\label{eq:simons-commutator}
 \mathfrak S_{ijkl}
 :=h_{ij}(h^2)_{kl}-h_{kl}(h^2)_{ij}
   -\kappa^2(g_{ij}h_{kl}-g_{kl}h_{ij}),
\end{equation}
 which is symmetric in each index pair and satisfies
\begin{equation*}
 \mathfrak S_{ijkl}=-\mathfrak S_{klij}.
\end{equation*}
In an orthonormal principal frame, its only possibly nonzero components are
\begin{equation}\label{eq:simons-commutator-principal}
 \mathfrak S_{iijj}
 =(\lambda_j-\lambda_i)(\lambda_i\lambda_j-\kappa^2),
\end{equation}
and consequently
\begin{equation}\label{eq:simons-commutator-norm}
 |\mathfrak S|^2
 =\sum_{i,j=1}^n
 (\lambda_i-\lambda_j)^2(\lambda_i\lambda_j-\kappa^2)^2.
\end{equation}

\begin{lemma}[Coercivity of the Simons commutator]
\label{lem:simons-commutator-coercivity}
For every $\eta>0$ there is
$\gamma=\gamma(n,\alpha_0,\eta)>0$ with the following property.  Uniformly for
$0\leq\tau\leq2$ and $0\leq\kappa\leq1$, every principal-curvature vector satisfying
$H>0$, $H\geq n\kappa$ and the quantitative $\tau$-bi-Ricci pinching condition \eqref{eq:uniform-taubiric} also satisfies
\begin{equation*}
 |\mathfrak S|^2\geq\gamma H^6
\end{equation*}
whenever
\begin{equation}\label{eq:acylindrical-region}
 |A|^2-\frac1{n-1}H^2\geq\eta H^2.
\end{equation}
\end{lemma}

\begin{proof}
All the curvature inequalities in the assertion are homogeneous under
simultaneous scaling of
$(\lambda_1,\ldots,\lambda_n,\kappa)$.  We therefore divide by $H$ and write
\begin{equation*}
 \mu_i=\frac{\lambda_i}{H},
 \qquad
 s=\frac{\kappa}{H}.
\end{equation*}
The pointwise argument in Proposition~\ref{prop:uniform-two-convexity} gives
$|\mu_i|\leq1$, while $H\geq n\kappa$ gives $0\leq s\leq1/n$.  We also have
$\sum_i\mu_i=1$.  The normalized form of
\eqref{eq:acylindrical-region} is
\begin{equation}\label{eq:normalized-acylindrical-defect}
 \sum_i\mu_i^2\geq\frac1{n-1}+\eta.
\end{equation}
After division by $H^2$, the quantitative pinching condition \eqref{eq:uniform-taubiric} becomes
\begin{equation}\label{eq:normalized-tau-pinching}
 \mu_i+\mu_j-\mu_i^2-\mu_j^2-\tau\mu_i\mu_j
 -c_{n,\tau}s^2
 \geq\alpha_0|\mu_p-\mu_q|
\end{equation}
for every $i\ne j$ and all $p,q$.

We now show that the normalized Simons commutator cannot vanish under these
conditions.  If it did, \eqref{eq:simons-commutator-principal} would give
\begin{equation}\label{eq:simons-zero}
 (\mu_i-\mu_j)(\mu_i\mu_j-s^2)=0
 \qquad\text{for every }i,j.
\end{equation}
If all $\mu_i$ were equal, then
$\mu_i=1/n$ for every $i$, and hence
\begin{equation*}
 \sum_i\mu_i^2=\frac1n<\frac1{n-1},
\end{equation*}
contrary to \eqref{eq:normalized-acylindrical-defect}.

Suppose that the $\mu_i$ are not all equal.  It follows from
\eqref{eq:simons-zero} that there are exactly two distinct values.
Indeed, the product of any two distinct values must equal $s^2$.  If three
distinct values $x,y,z$ occurred, then $xy=xz=yz=s^2$.  The first two
equalities would give $x(y-z)=0$, and hence $x=0$.  It would follow that
$s=0$ and $yz=0$, which is impossible for three distinct values.  Denote the
two values by $a>b$.  Since $ab=s^2\geq0$ and their weighted sum is one,
\begin{equation*}
 a>b\geq0.
\end{equation*}
Let $a$ have multiplicity $m$ and $b$ multiplicity $q=n-m$.

Suppose first that $q=1$.  Then $m=n-1$ and
$(n-1)a+b=1$.  A direct computation gives
\begin{equation*}
 1-(n-1)\sum_i\mu_i^2
 =b\bigl(2(n-1)a-(n-2)b\bigr)\geq0,
\end{equation*}
where the last inequality follows from $a>b\geq0$.  Thus
$\sum_i\mu_i^2\leq1/(n-1)$, again contradicting
\eqref{eq:normalized-acylindrical-defect}.

We may therefore assume that $q\geq2$.  Choose two distinct indices
$i,j$ in the $b$-eigenspace.  Since
\begin{equation*}
 (n-q)a+qb=1,
 \qquad
 s^2=ab,
 \qquad
 c_{n,\tau}=2(n-1)-\tau,
\end{equation*}
the left-hand side of
\eqref{eq:normalized-tau-pinching} for this pair is
\begin{align*}
 &2b-(2+\tau)b^2-c_{n,\tau}s^2\\
 =&(2q-2-\tau)b(b-a)\leq0.
\end{align*}
Here $2q-2-\tau\geq0$ because $q\geq2$ and $0\leq\tau\leq2$.
On the other hand, in the right-hand side of
\eqref{eq:normalized-tau-pinching} we may choose one index from the
$a$-eigenspace and one from the $b$-eigenspace.  The quantitative pinching
then gives the strictly positive lower bound
$\alpha_0(a-b)$.  This is a contradiction.

The variables satisfying $|\mu_i|\leq1$, $0\leq s\leq1/n$,
$0\leq\tau\leq2$, $\sum_i\mu_i=1$, and the two normalized inequalities form
a compact set.  If this set is empty, the conclusion is immediate.  Otherwise,
the preceding argument shows that
\begin{equation*}
 \sum_{i,j}
 (\mu_i-\mu_j)^2(\mu_i\mu_j-s^2)^2
\end{equation*}
has no zero on this set.  Its minimum is therefore positive and depends only
on $n$, $\alpha_0$, and $\eta$.  Multiplying back by $H^6$ proves the lemma.
\end{proof}

\begin{proposition}[$\tau$-bi-Ricci Poincar\'e inequality]
\label{prop:tau-poincare}
For every $\eta>0$ there exists
$\gamma=\gamma(n,\alpha_0,\eta)>0$ with the following property.
Let
\begin{equation*}
 X:M^n\longrightarrow\Hh^{n+1}_{\kappa}
\end{equation*}
be a smooth closed two-sided hypersurface immersion satisfying
$H>0$, $H\geq n\kappa$, and \eqref{eq:uniform-taubiric}.  Suppose that
$u\in W^{1,2}(M)$ vanishes almost everywhere outside the set
\eqref{eq:acylindrical-region}.  Then, for every $\varrho\geq1$,
\begin{equation}\label{eq:tau-poincare}
 \gamma\int_Mu^2H^2\,d\mu
 \leq
 \int_M\left(
 \varrho^{-1}|\nabla u|^2
 +\varrho u^2\frac{|\nabla A|^2}{H^2}
 \right)d\mu.
\end{equation}
The constant $\gamma$ is uniform for
$0\leq\tau\leq2$ and $0\leq\kappa\leq1$.
\end{proposition}

\begin{proof}
Replacing $u$ by $|u|$ does not change either side of
\eqref{eq:tau-poincare}, so we may assume that $u\geq0$.  Since $M$ is
closed and $H$ is smooth and strictly positive, the function
\begin{equation*}
 \phi:=u^2H^{-4}
\end{equation*}
belongs to $W^{1,1}(M)$.  This weight is chosen because the coercive bound
of Lemma~\ref{lem:simons-commutator-coercivity} then makes
$\phi|\mathfrak S|^2$ control $u^2H^2$, while its derivative produces exactly
the powers of $H$ needed after integration by parts.  We have
\begin{equation}\label{eq:poincare-weight-gradient}
 \nabla\phi
 =2uH^{-4}\nabla u-4u^2H^{-5}\nabla H 
\end{equation}
 almost everywhere on $M$.  Since $u=0$ almost everywhere outside the
 acylindrical set \eqref{eq:acylindrical-region},
 Lemma~\ref{lem:simons-commutator-coercivity} gives
\begin{equation}\label{eq:poincare-coercivity-start}
 \gamma_0\int_M u^2H^2\,d\mu
 \leq\int_M\phi|\mathfrak S|^2\,d\mu
\end{equation}
for some $\gamma_0=\gamma_0(n,\alpha_0,\eta)>0$.

Since
$\mathfrak S$ is symmetric in each index pair, the normalized
symmetrizations in \eqref{eq:full-simons} disappear after contraction.
Using also $\mathfrak S_{klij}=-\mathfrak S_{ijkl}$ and relabeling the two
index pairs, we obtain
\begin{align}
 |\mathfrak S|^2
 &=\mathfrak S_{ijkl}
   \bigl(\nabla_{(i}\nabla_{j)}h_{kl}
         -\nabla_{(k}\nabla_{l)}h_{ij}\bigr)\notag\\
 &=\mathfrak S_{ijkl}\nabla_i\nabla_jh_{kl}
   -\mathfrak S_{klij}\nabla_i\nabla_jh_{kl}\notag\\
 &=2\mathfrak S_{ijkl}\nabla_i\nabla_jh_{kl}.
 \label{eq:simons-contracted-full}
\end{align}
The tensor field $\mathfrak S$ and the derivatives of $h$ are smooth, whereas
$\phi\in W^{1,1}(M)$.  Thus the weak integration-by-parts formula on the
closed manifold $M$ applies directly to \eqref{eq:simons-contracted-full}:
\begin{align}
 \int_M\phi|\mathfrak S|^2\,d\mu
 &=-2\int_M\nabla_i(\phi\mathfrak S_{ijkl})
                  \nabla_jh_{kl}\,d\mu.\label{eq:poincare-ibp}
\end{align}
This avoids any support-preserving smooth approximation of $u$.

The elementary bounds from Proposition~\ref{prop:uniform-two-convexity}
give $|A|\leq\sqrt n\,H$, and the hypothesis $H\geq n\kappa$ gives
$\kappa\leq H/n$.  From the explicit formula
\eqref{eq:simons-commutator} and
$\nabla_iH=g^{kl}\nabla_i h_{kl}$ we therefore obtain
\begin{equation}\label{eq:commutator-derivative-bounds}
 |\mathfrak S|\leq C H^3,
 \qquad
 |\nabla\mathfrak S|\leq C H^2|\nabla A|,
 \qquad
 |\nabla H|\leq\sqrt n\,|\nabla A|,
\end{equation}
where $C=C(n)$.  Expanding the derivative in
\eqref{eq:poincare-ibp}, using
\eqref{eq:poincare-weight-gradient} and
\eqref{eq:commutator-derivative-bounds}, gives
\begin{align*}
 \int_M\phi|\mathfrak S|^2\,d\mu
 &\leq 2\int_M|\nabla\phi|\,|\mathfrak S|\,|\nabla A|\,d\mu
       +2\int_M\phi|\nabla\mathfrak S|\,|\nabla A|\,d\mu\\
 &\leq C\int_M\left(
 |u|H^{-1}|\nabla u||\nabla A|
 +u^2H^{-2}|\nabla A|^2\right)d\mu.
\end{align*}
For $\varrho\geq1$, Young's inequality yields pointwise
\begin{equation*}
 |u|H^{-1}|\nabla u||\nabla A|
 \leq\frac1{2\varrho}|\nabla u|^2
      +\frac{\varrho}{2}u^2\frac{|\nabla A|^2}{H^2},
\end{equation*}
and the remaining $u^2H^{-2}|\nabla A|^2$ term is bounded by the same
quantity with coefficient $\varrho$.  Consequently,
\begin{equation*}
 \int_M\phi|\mathfrak S|^2\,d\mu
 \leq \frac{C}{\varrho}\int_M|\nabla u|^2\,d\mu
       +C\varrho\int_Mu^2\frac{|\nabla A|^2}{H^2}\,d\mu.
\end{equation*}
Combining this estimate with \eqref{eq:poincare-coercivity-start} and
replacing $\gamma_0/C$ by $\gamma$ proves \eqref{eq:tau-poincare}.
\end{proof}

\subsection{Integral estimates and Stampacchia iteration}

With Proposition~\ref{prop:tau-poincare} established, the remaining argument
is the standard $L^p$--Stampacchia scheme.  We follow
\cite{HS09}*{Lemmas~5.4--5.6 and the proof of Theorem~4.6} and
\cite{LN21}*{Section~4.1}.  We record the estimates needed to control the
additional hyperbolic term and to keep the constants uniform in $\kappa$ and
in the original curvature scale.

Fix $\eta>0$ and $\sigma\in(0,1/2)$.  Define
\begin{equation*}
  f_{\sigma,\eta}
  :=H^{\sigma-2}\left(|A|^2-\left(\frac{1}{n-1}+\eta\right)H^2\right),
  \qquad f_+:=\max\{f_{\sigma,\eta},0\}.
\end{equation*}
When no confusion is possible we write $f=f_{\sigma,\eta}$.  On the support of
$f_+$,
\begin{equation*}
  |A|^2-\frac{1}{n-1}H^2\geq \eta H^2,
  \qquad 0<f_+\leq C(n)H^\sigma.
\end{equation*}
The upper bound follows from $|A|^2\le nH^2$.  It is convenient to set
\begin{equation}\label{eq:relative-gradient-tensor}
 \mathcal D_{ikl}:=H\nabla_i h_{kl}-(\nabla_iH)h_{kl}.
\end{equation}
The identity
$\nabla_i h_{kl}=H^{-1}\mathcal D_{ikl}
+H^{-1}(\nabla_iH)h_{kl}$ and \eqref{eq:kato-hs} give
$|\nabla H|^2\leq8\beta^{-2}H^{-2}|\mathcal D|^2$, where
$\beta=\beta(n,\alpha_0)$ is the constant in
Proposition~\ref{prop:uniform-two-convexity}.  Using also
$|A|^2\leq nH^2$, we obtain
\begin{equation}\label{eq:D-controls-gradient}
 \frac{|\nabla A|^2}{H^2}
 \leq 2\frac{|\mathcal D|^2}{H^4}
 +2\frac{|A|^2}{H^4}|\nabla H|^2
 \leq C(n,\alpha_0)\frac{|\mathcal D|^2}{H^4}.
\end{equation}

The Euclidean part of the following evolution identity is the
Huisken--Sinestrari calculation \cite{HS09}*{(5.6)}.  We retain the
constant-curvature term explicitly.  Compare also
\cite{LN21}*{Lemma~4.5}.

\begin{lemma}
\begin{align}
 (\partial_t-\Delta)f_{\sigma,\eta}
 &=2(1-\sigma)\left\langle
       \nabla f_{\sigma,\eta},\frac{\nabla H}{H}\right\rangle
 -\frac{2}{H^{4-\sigma}}
  |H\nabla_i h_{kl}-(\nabla_iH)h_{kl}|^2 \notag\\
 &\quad-\sigma(1-\sigma)f_{\sigma,\eta}
          \frac{|\nabla H|^2}{H^2}
 +\sigma f_{\sigma,\eta}(|A|^2-n\kappa^2)\notag\\
 &\quad+4n\kappa^2H^{\sigma-2}
       \left(|A|^2-\frac1nH^2\right).
 \label{eq:f-evol}
\end{align}
\end{lemma}
\begin{proof}
Set
\begin{equation*}
 Z =|A|^2-
 \left(\frac{1}{n-1}+\eta\right)H^2,
\end{equation*}
so that $f_{\sigma,\eta}=H^{\sigma-2}Z$.
Equations \eqref{eq:H2-evol} and \eqref{eq:A2-evol} give
\begin{align*}
 (\partial_t-\Delta)Z
 ={}&
 -2\left(
 |\nabla A|^2-
 \left(\frac{1}{n-1}+\eta\right)|\nabla H|^2
 \right)\\
 &+2(|A|^2-n\kappa^2)Z
 +4n\kappa^2
 \left(
 |A|^2-\frac1nH^2
 \right).
\end{align*}
The evolution equation \eqref{eq:H-evol} also gives
\begin{align*}
 (\partial_t-\Delta)H^{\sigma-2}
 ={}&
 (\sigma-2)H^{\sigma-2}
 (|A|^2-n\kappa^2)\\
 &-(\sigma-2)(\sigma-3)
 H^{\sigma-4}|\nabla H|^2.
\end{align*}
Moreover, from the definition \eqref{eq:relative-gradient-tensor},
\begin{align*}
 |D|^2 ={}&
 H^2|\nabla A|^2
 -H\langle\nabla Z,\nabla H\rangle\\
 &+ \left[
 Z-\left(\frac{1}{n-1}+\eta\right)H^2
 \right]|\nabla H|^2.
\end{align*}
Applying the product rule to
$f_{\sigma,\eta}=H^{\sigma-2}Z$ and using this identity yields
\eqref{eq:f-evol}.
\end{proof}

\begin{proposition}[Uniform $L^p$ bound]\label{prop:lp-bound}
For every $\eta>0$ there are constants $c_4,c_5>0$, depending only on $n$,
$\alpha_0$, and $\eta$, such that if
\begin{equation*}
 p\geq c_4,
 \qquad
 0<\sigma\leq\min\left\{\frac12,c_5p^{-1/2}\right\},
\end{equation*}
then, for $0\leq t<T$,
\begin{equation}\label{eq:lp-bound}
 \int_{M_t}f_+^p\,d\mu
 \leq\int_{M_0}f_+^p\,d\mu
 +C(n,\alpha,\eta,\sigma,p)|M_0|t.
\end{equation}
The constant is independent of $\kappa$ and $T$.
\end{proposition}

\begin{proof}
We work on the set $\{f>0\}$ and use a smooth convex approximation of the
positive part.  Multiplying \eqref{eq:f-evol} by $pf_+^{p-1}$, integrating by
parts, and using \eqref{s2.evl-dmu}, we obtain
\begin{align*}
 \frac{d}{dt}\int f_+^p d\mu
 &\leq-p(p-1)\int f_+^{p-2}|\nabla f|^2d\mu
 +2p(1-\sigma)\int f_+^{p-1}
     \left\langle\nabla f,\frac{\nabla H}{H}\right\rangle d\mu\\
 &\quad-2p\int f_+^{p-1}H^{\sigma-4}|\mathcal D|^2d\mu
 +p\sigma\int f_+^p(|A|^2-n\kappa^2)d\mu\\
 &\quad+Cp\kappa^2\int H^\sigma f_+^{p-1}d\mu.
\end{align*}
The nonpositive term
$-p\sigma(1-\sigma)\int f_+^p|\nabla H|^2/H^2$ and the contribution
$-\int H^2f_+^p$ from the evolving measure have only been discarded.

By \eqref{eq:kato-hs},
$|\nabla H|/H\leq C|\mathcal D|/H^2$.  Consequently,
\begin{align*}
 &2p(1-\sigma)\int f_+^{p-1}
   \left|\left\langle\nabla f,\frac{\nabla H}{H}\right\rangle\right|d\mu\\
 \leq&\frac{p(p-1)}4\int f_+^{p-2}|\nabla f|^2d\mu
 +\frac{Cp}{p-1}\int f_+^p\frac{|\mathcal D|^2}{H^4}d\mu\\
 \leq&\frac{p(p-1)}4\int f_+^{p-2}|\nabla f|^2d\mu
 +C\int f_+^{p-1}H^{\sigma-4}|\mathcal D|^2d\mu.
\end{align*}
Here the last step uses $f_+\leq CH^\sigma$.  For all sufficiently large
$p$, the final term is absorbed into the negative $\mathcal D$-term.

The reaction term is bounded by
\begin{equation*}
 p\sigma\int f_+^p(|A|^2-n\kappa^2)
 \leq Cp\sigma\int H^2f_+^p.
\end{equation*}
For the remaining space-form term, Young's inequality gives, for any
$\theta>0$,
\begin{equation*}
 H^\sigma f_+^{p-1}
 \leq\theta H^2f_+^p
 +C_{\theta,p,\sigma}H^{p\sigma-2(p-1)}.
\end{equation*}
The last exponent is nonpositive for $p\geq2$ and
$0<\sigma<1/2$, and $H\geq\alpha_2$ under
\eqref{eq:normalized-space-form-setting}.  Taking $\theta=\sigma$ and using
$0\leq\kappa\leq1$, we find constants $c_0>0$ and $p_0>2$, depending only
on $n$ and $\alpha_0$, such that, for every $p\geq p_0$,
\begin{align}
 \frac{d}{dt}\int_{M_t}f_+^p\,d\mu
 &\leq-c_0p(p-1)\int_{M_t}f_+^{p-2}|\nabla f|^2\,d\mu\notag\\
 &\quad-c_0p\int_{M_t}f_+^{p-1}H^{\sigma-4}
                  |\mathcal D|^2\,d\mu\notag\\
 &\quad+Cp\sigma\int_{M_t}H^2f_+^p\,d\mu
       +C(n,\alpha,p,\sigma)|M_t|.
 \label{eq:first-lp}
\end{align}

To estimate the first positive term in \eqref{eq:first-lp}, we use the Poincar\'e estimate in Proposition~\ref{prop:tau-poincare}. Set $u=f_+^{p/2}$, which vanishes
outside the set \eqref{eq:acylindrical-region}, and apply
Proposition~\ref{prop:tau-poincare} with $\varrho=p^{1/2}$.  Since
\begin{equation*}
 |\nabla u|^2=\frac{p^2}{4}f_+^{p-2}|\nabla f|^2,
\end{equation*}
while \eqref{eq:D-controls-gradient} and $f_+\leq CH^\sigma$ give
\begin{align*}
 \int f_+^p\frac{|\nabla A|^2}{H^2}\,d\mu
 &\leq C\int f_+^p\frac{|\mathcal D|^2}{H^4}\,d\mu\\
 &\leq C\int f_+^{p-1}H^{\sigma-4}|\mathcal D|^2\,d\mu,
\end{align*}
we obtain, for every $p\geq2$,
\begin{align}
 \int_{M_t}H^2f_+^p\,d\mu
 &\leq Cp^{3/2}\int_{M_t}f_+^{p-2}|\nabla f|^2\,d\mu
 +Cp^{1/2}\int_{M_t}f_+^{p-1}H^{\sigma-4}
              |\mathcal D|^2\,d\mu,
 \label{eq:poincare-cylindrical}
\end{align}
where $C=C(n,\alpha_0,\eta)$ is independent of $\kappa$.

Finally, substituting \eqref{eq:poincare-cylindrical} into
\eqref{eq:first-lp}, we obtain
\begin{align*}
 \frac{d}{dt}\int_{M_t}f_+^p\,d\mu
 &\leq
 -\bigl(c_0p(p-1)-C\sigma p^{5/2}\bigr)
 \int_{M_t}f_+^{p-2}|\nabla f|^2\,d\mu\\
 &\quad
 -\bigl(c_0p-C\sigma p^{3/2}\bigr)
 \int_{M_t}f_+^{p-1}H^{\sigma-4}
 |\mathcal D|^2\,d\mu\\
 &\quad
 +C(n,\alpha,\eta,p,\sigma)|M_t|.
\end{align*}
Choose $c_4\geq\max\{p_0,2\}$ sufficiently large and then
$0<c_5\leq1/2$ sufficiently small, depending only on $n$, $\alpha_0$, and
$\eta$, so that
\begin{equation*}
 c_0p(p-1)-C\sigma p^{5/2}\geq0,
 \qquad
 c_0p-C\sigma p^{3/2}\geq0
\end{equation*}
whenever
\begin{equation*}
 p\geq c_4,
 \qquad
 0<\sigma\leq c_5p^{-1/2}.
\end{equation*}
It follows that
\begin{equation*}
 \frac{d}{dt}\int_{M_t}f_+^p\,d\mu
 \leq C(n,\alpha,\eta,p,\sigma)|M_t|.
\end{equation*}
Since the area is nonincreasing along the smooth mean curvature flow,
integration over $[0,t]$ yields  \eqref{eq:lp-bound}.
\end{proof}

\begin{lemma}[Truncated spacetime estimate]\label{lem:truncated-energy}
Fix $\eta>0$ under \eqref{eq:normalized-space-form-setting}, and choose $ r>{(n+2)}/{2}$.  
There is $p_0=p_0(n,\alpha,\eta,r)$ such that the following holds.  If
$p\geq p_0$, $\sigma=p^{-1}$, $f=f_{\sigma,\eta}$, and
\begin{equation*}
 k\geq k_0:=1+\sup_{M_0}f_+,
 \qquad f_k=(f-k)_+,
\end{equation*}
then
\begin{align}
 &\sup_{0\leq s<T}\int_{M_s}f_k^p\,d\mu
 +\int_0^T\!\!\int_{M_s}
   \left(f_k^{p-2}|\nabla f|^2
   +f_k^{p-1}H^{\sigma-4}|\mathcal D|^2
   +H^2f_k^p\right)d\mu\,ds\notag\\
 &\hspace{35mm}\leq
 C\int_0^T\!\!\int_{\{f(\cdot,s)>k\}}
   \bigl(H^2f_+^p+1\bigr)d\mu\,ds,
 \label{eq:truncated-energy}
\end{align}
where $C=C(n,\alpha,\eta,p)$.  Moreover,
\begin{equation}\label{eq:weighted-integrability}
 \int_0^T\!\!\int_{M_s}\bigl(H^2f_+^p+1\bigr)^r\,d\mu\,ds
 \leq C(n,\alpha,\eta,p,r).
\end{equation}
\end{lemma}

\begin{proof}
Choose $p_0$ sufficiently large so that both pairs
\begin{equation*}
 (p,\sigma)=\left(p,\frac1p\right),
 \qquad
 (pr,\sigma+2/p)=\left(pr,\frac3p\right)
\end{equation*}
satisfy the assumptions of Proposition~\ref{prop:lp-bound}.  This is possible
because $3/p\leq c_5(pr)^{-1/2}$ 
for all sufficiently large $p$.

We first prove \eqref{eq:truncated-energy}.  Since
$k\geq1+\sup_{M_0}f_+$, we have $f_k(\cdot,0)=0$.  Using a smooth convex
approximation of $s\mapsto(s-k)_+$, multiply \eqref{eq:f-evol} by
$pf_k^{p-1}$ and integrate over $M_t$.  On $\{f(\cdot,t)>k\}$ one has
$f=f_+$, $\nabla f_k=\nabla f$, and
$f_k\leq f_+\leq C H^\sigma$.  Estimating the mixed gradient term exactly
as in the derivation of \eqref{eq:first-lp}, and retaining the contribution
$-\int H^2f_k^p\,d\mu$ from the evolution of the measure, we obtain
\begin{align*}
 \frac{d}{dt}\int_{M_t}f_k^p\,d\mu
 &\leq
 -cp^2\int_{M_t}f_k^{p-2}|\nabla f|^2\,d\mu\\
 &\quad
 -cp\int_{M_t}f_k^{p-1}H^{\sigma-4}
 |\mathcal D|^2\,d\mu
 -\int_{M_t}H^2f_k^p\,d\mu\\
 &\quad
 +Cp\sigma\int_{\{f(\cdot,t)>k\}}H^2f_+^p\,d\mu
 +Cp\kappa^2\int_{\{f(\cdot,t)>k\}}H^\sigma f_k^{p-1}\,d\mu,
\end{align*}
after increasing $p_0$ if necessary.  Here we have also used
\begin{equation*}
 f f_k^{p-1}\leq f_+^p
 \qquad\text{on }\{f(\cdot,t)>k\}
\end{equation*}
to estimate the reaction term.

Since $\sigma=1/p$, Young's inequality gives
\begin{equation*}
 Cp\kappa^2H^{1/p}f_k^{p-1}
 \leq
 \frac12H^2f_k^p
 +C(n,p)\kappa^{2p}H^{3-2p}.
\end{equation*}
Under the normalized assumptions,
$0\leq\kappa\leq1$ and $H\geq\alpha_2$, while $p\geq2$. Hence
\begin{equation*}
 \kappa^{2p}H^{3-2p}\leq C(n,\alpha_2,p).
\end{equation*}
Therefore,
\begin{align*}
 \frac{d}{dt}\int_{M_t}f_k^p\,d\mu
 &+cp^2\int_{M_t}f_k^{p-2}|\nabla f|^2\,d\mu\\
 &+cp\int_{M_t}f_k^{p-1}H^{\sigma-4}
 |\mathcal D|^2\,d\mu
 +\frac12\int_{M_t}H^2f_k^p\,d\mu\\
 &\leq
 C\int_{\{f(\cdot,t)>k\}}\bigl(H^2f_+^p+1\bigr)\,d\mu,
\end{align*}
where we used $p\sigma=1$.  Integrating over $[0,t]$, taking the supremum
over $t<T$, and then letting $t\uparrow T$ proves
\eqref{eq:truncated-energy}.

It remains to prove \eqref{eq:weighted-integrability}.  From the definition
of $f_{\sigma,\eta}$, $ f_{\sigma+2/p,\eta}
 =H^{2/p}f_{\sigma,\eta}$.  
Since $\sigma=1/p$, we have
\begin{equation*}
 H^{2r}(f_{\sigma,\eta})_+^{pr}
 =\bigl(f_{3/p,\eta}\bigr)_+^{pr}.
\end{equation*}
Applying Proposition~\ref{prop:lp-bound} with exponent $pr$ and parameter
$3/p$, we obtain, for every $0\leq s<T$,
\begin{equation*}
 \int_{M_s}\bigl(f_{3/p,\eta}\bigr)_+^{pr}\,d\mu
 \leq
 \int_{M_0}\bigl(f_{3/p,\eta}\bigr)_+^{pr}\,d\mu
 +C(n,\alpha,\eta,p,r)|M_0|\,s.
\end{equation*}
On the initial hypersurface, $ H_0^2\leq n|A_0|^2\leq n$  
and
$(f_{3/p,\eta})_+\leq C(n)H^{3/p}$.  Together with
$|M_0|\leq\alpha_1$, this gives
\begin{equation*}
 \int_{M_0}\bigl(f_{3/p,\eta}\bigr)_+^{pr}\,d\mu
 \leq C(n,\alpha_1,p,r).
\end{equation*}
Integrating in time and using $T\leq\frac n2\alpha_2^{-2}$ 
from Lemma~\ref{lem:time-upper}, we conclude that
\begin{equation*}
 \int_0^T\!\!\int_{M_s}
 H^{2r}(f_{\sigma,\eta})_+^{pr}\,d\mu\,ds
 \leq C(n,\alpha,\eta,p,r).
\end{equation*}
Finally,
\begin{equation*}
 (H^2f_+^p+1)^r
 \leq C(r)\bigl(H^{2r}f_+^{pr}+1\bigr),
\end{equation*}
and
\begin{equation*}
 \int_0^T|M_s|\,ds\leq T|M_0|.
\end{equation*}
This proves \eqref{eq:weighted-integrability}.
\end{proof}

\begin{proof}[Proof of Theorem~\ref{thm:cylindrical}]
It suffices to consider $0<\eta\leq1$.  We first work under the normalized
assumptions \eqref{eq:normalized-space-form-setting}.  Fix $r>(n+2)/2$.  
Apply Lemma~\ref{lem:truncated-energy} with $\eta/2$, and choose
$p=p(n,\alpha,\eta,r)$ sufficiently large so that all the conclusions of that
lemma hold.  Set
\begin{equation*}
 \sigma=\frac1p,
 \qquad
 f=f_{\sigma,\eta/2},
 \qquad
 k_0=1+\sup_{M_0}f_+,
 \qquad
 f_k=(f-k)_+.
\end{equation*}
For $k\geq k_0$, define
\begin{equation*}
 \mathcal A(k)
 :=
 \int_0^T
 \bigl|\{x\in M_t:f(x,t)>k\}\bigr|\,dt.
\end{equation*}
We shall prove that $\mathcal A(k)=0$ for some uniformly bounded $k$.

Let
\begin{equation*}
 q_0:=\frac{n+2}{n},
 \qquad
 w:=f_k^{p/2}.
\end{equation*}
Since the ambient hyperbolic space is complete, simply connected, and
nonpositively curved, the Michael--Simon Sobolev inequality in \cites{HoffmanSpruck74,HoffmanSpruck75} gives on each time slice
\begin{equation*}
 \left(
 \int_{M_t}w^{\frac{2n}{n-2}}\,d\mu
 \right)^{\frac{n-2}{n}}
 \leq
 C\int_{M_t}\bigl(|\nabla w|^2+H^2w^2\bigr)\,d\mu.
\end{equation*}
Applying this inequality to $w$, integrating in time, and using H\"older's
inequality together with
$|\nabla w|^2=p^2f_k^{p-2}|\nabla f|^2/4$, we obtain from
Lemma~\ref{lem:truncated-energy}
\begin{align}
 \int_0^T\!\!\int_{M_t}f_k^{pq_0}\,d\mu\,dt
 &\leq
 C\left[
 \int_0^T
 \int_{\{f(\cdot,t)>k\}}
 \bigl(H^2f_+^p+1\bigr)\,d\mu\,dt
 \right]^{q_0}.
 \label{eq:stampacchia-parabolic-estimate}
\end{align}
The weighted integrability estimate \eqref{eq:weighted-integrability} and
H\"older's inequality on spacetime then give
\begin{equation*}
 \int_0^T\!\!\int_{M_t}f_k^{pq_0}\,d\mu\,dt
 \leq
 C\,\mathcal A(k)^{q_0(1-\frac1r)}.
\end{equation*}

Now let $h>k\geq k_0$.  On the set $\{f(\cdot,t)>h\}$ one has
$f_k\geq h-k$, and hence
\begin{equation*}
 (h-k)^{pq_0}\mathcal A(h)
 \leq
 \int_0^T\!\!\int_{M_t}f_k^{pq_0}\,d\mu\,dt.
\end{equation*}
Consequently,
\begin{equation}\label{eq:stampacchia-recursion}
 (h-k)^{pq_0}\mathcal A(h)
 \leq
 C\mathcal A(k)^{1+\mu},
 \qquad h>k\geq k_0,
\end{equation}
where $ \mu
 := q_0\left(1-\frac1r\right)-1>0$  
because $r>(n+2)/2$.

The function $\mathcal A$ is nonincreasing.  The normalized time bound and
area monotonicity give
\begin{equation*}
 \mathcal A(k_0)
 \leq
 \int_0^T|M_t|\,dt
 \leq
 T|M_0|
 \leq
 C(n,\alpha).
\end{equation*}
On the initial hypersurface, $H_0^2\leq n|A_0|^2\leq n$.  Since
$f_+\leq C(n)H^\sigma$, it follows that
\begin{equation*}
 k_0
 =1+\sup_{M_0}f_+
 \leq C(n).
\end{equation*}
The Stampacchia lemma
\cite{KS80}*{Chapter~II, Lemma~B.1}, applied to
\eqref{eq:stampacchia-recursion}, therefore yields a level $k_*$ such that
\begin{equation*}
 \mathcal A(k_*)=0,
 \qquad
 k_0\leq k_*,
 \qquad
 k_*\leq C(n,\alpha,\eta).
\end{equation*}
If $f(x,t)>k_*$ at some $t>0$, continuity gives a spacetime neighborhood
of positive measure on which $f>k_*$, contradicting
$\mathcal A(k_*)=0$.  At $t=0$, the same conclusion follows from
$k_*\geq k_0>\sup_{M_0}f_+$.  Thus
\begin{equation}\label{eq:normalized-cylindrical-f-bound}
 f_{\sigma,\eta/2}
 \leq C(n,\alpha,\eta)
\end{equation}
throughout the normalized smooth flow.

We now return to the original scale.  Under the parabolic rescaling of
Subsection~\ref{subsec:space-form-normalization}, $ \widehat f_{\sigma,\eta/2}
 =R^\sigma f_{\sigma,\eta/2}$.  
Hence \eqref{eq:normalized-cylindrical-f-bound} gives
\begin{equation*}
 |A|^2
 -\left(\frac1{n-1}+\frac\eta2\right)H^2
 \leq
 C R^{-\sigma}H^{2-\sigma}.
\end{equation*}
Young's inequality, with conjugate exponents
$2/(2-\sigma)$ and $2/\sigma$, yields
\begin{equation*}
 C R^{-\sigma}H^{2-\sigma}
 \leq
 \frac\eta2 H^2+C_\eta R^{-2}.
\end{equation*}
Therefore
\begin{equation*}
 |A|^2-\frac1{n-1}H^2
 \leq
 \eta H^2+C_\eta R^{-2},
\end{equation*}
which proves \eqref{eq:cylindrical}.
\end{proof}

\begin{corollary}[Convexity estimate]\label{cor:convexity}
For every $\varepsilon>0$ there exists
$C_\varepsilon=C(n,\alpha,\varepsilon)>0$ such that
\begin{equation}\label{eq:convexity-estimate}
 \lambda_1\geq-\varepsilon H-C_\varepsilon R^{-1}
\end{equation}
at every point of the smooth flow.
\end{corollary}

\begin{proof}
There is nothing to prove where $\lambda_1\geq0$.  If $\lambda_1<0$, then
Cauchy's inequality gives
\begin{align*}
 |A|^2-\frac1{n-1}H^2
 &\geq
 \lambda_1^2+\frac{(H-\lambda_1)^2}{n-1}
 -\frac1{n-1}H^2\\
 &=\frac{n\lambda_1^2-2H\lambda_1}{n-1}\\
 &\geq\frac{n}{n-1}\lambda_1^2.
\end{align*}
Apply Theorem~\ref{thm:cylindrical} with
$\eta=n\varepsilon^2/(n-1)$ and take square roots to obtain
\begin{equation*}
 -\lambda_1
 \leq
 \varepsilon H+
 \left(
 \frac{n-1}{n}
 C_{\frac{n}{n-1}\varepsilon^2}
 \right)^{1/2}R^{-1},
\end{equation*}
which proves \eqref{eq:convexity-estimate} after renaming the constant.
\end{proof}

\section{Derivative estimates and neck detection}
\label{sec:gradient-estimates}

Let $M_t$, $t\in[0,T)$, be a smooth mean curvature flow with
\begin{equation*}
 M_0\in\CC_{\tau,\kappa}(R,\alpha),
 \qquad
 |A|^2\leq R^{-2}\quad\text{on }M_0,
\end{equation*}
where $\alpha=(\alpha_0,\alpha_1,\alpha_2)$.  In this section we derive pointwise estimates for the first and higher
derivatives of the second fundamental form and then use them, together with
the cylindrical estimate, to detect necks at sufficiently large curvature.
For the derivative estimates compare \cite{HS09}*{Section~6} and the
spherical space-form analogue \cite{LN21}*{Sections~4.2--4.3}.  The neck
detection argument corresponds to \cite{HS09}*{Section~7} and
\cite{LN21}*{Section~4.4}.  The additional ambient-curvature terms in the
hyperbolic setting are of lower order and will be kept explicitly below.

Unless otherwise indicated, all constants in this section depend only on
$n$ and $\alpha$ and are uniform for
$\tau\in\mathcal I_n$, $0\leq\kappa\leq1$, $0<R\leq1$, and the maximal
smooth existence time $T$.  We shall repeatedly use $ \kappa^2\leq R^{-2}$.   
Set
\begin{equation*}
 \vartheta_n
 :=\frac12\left(
 \frac3{n+2}-\frac1{n-1}
 \right)>0.
\end{equation*}
The positivity of $\vartheta_n$ for $n\geq3$ is precisely the gap between the
Kato coefficient $3/(n+2)$ and the cylindrical coefficient $1/(n-1)$. It
provides the coercive term in the gradient estimate.

\subsection{The gradient estimate}

We first record the short-time smoothing estimates that initialize the
maximum-principle arguments.  Set $t_0=\tau_nR^2$.  The ODE comparison in the
proof of Lemma~\ref{lem:time-upper} gives
\begin{equation*}
 \sup_{M_t}|A|^2\leq C(n)R^{-2}
 \qquad\text{for }0\leq t\leq t_0.
\end{equation*}
Equations~\eqref{eq:A2-evol} and the first inequality in
\eqref{eq:derivative-evolutions} therefore imply, on $[0,t_0]$,
\begin{align*}
 (\partial_t-\Delta)|A|^2
 \leq &~-2|\nabla A|^2+CR^{-4},\\
 (\partial_t-\Delta)|\nabla A|^2
 \leq&~-2|\nabla^2A|^2+CR^{-2}|\nabla A|^2.
\end{align*}
Combining the two inequalities, we obtain
\begin{equation*}
 (\partial_t-\Delta)
 \bigl(t|\nabla A|^2+C(n)|A|^2\bigr)
 \leq C(n)R^{-4}.
\end{equation*}
At time zero the quantity in parentheses is at most $C(n)R^{-2}$, so the
maximum principle gives
$t_0\sup_{M_{t_0}}|\nabla A|^2\leq CR^{-2}$.  The same time-weighted
Bernstein induction, using the higher derivative analogues of
\eqref{eq:derivative-evolutions} and \eqref{eq:derivative-evolutions-2}, yields, for every integer $m\geq1$,
\begin{equation}\label{eq:short-time-smoothing}
 \sup_{M_{t_0}}|\nabla^mA|^2
 \leq C_m(n)R^{-2m-2}.
\end{equation}
The constants are uniform in $\kappa$ because $\kappa^2\leq R^{-2}$.

\begin{theorem}[Gradient estimate]\label{thm:gradient}
Let $M_0\in\CC_{\tau,\kappa}(R,\alpha)$ and assume that
$|A|^2\leq R^{-2}$ on $M_0$.  Then there exists
$C=C(n,\alpha)<\infty$ such that
\begin{equation}\label{eq:gradient-est}
 |\nabla A|^2
 \leq C\bigl(|A|^4+R^{-4}\bigr)
\end{equation}
on $M_t$ for every $t\in[\tau_nR^2,T)$.
\end{theorem}

\begin{proof}
Set $t_0=\tau_nR^2$.  By Lemma~\ref{lem:time-upper}, $T\geq2t_0$.
The case $m=1$ of \eqref{eq:short-time-smoothing} supplies the initial bound
for the argument below.

Apply Theorem~\ref{thm:cylindrical} with the fixed parameter $\vartheta_n$,
and denote the corresponding constant by $C_{\vartheta_n}$.  Enlarging it if
necessary, assume that $C_{\vartheta_n}\geq1$.  Set
\begin{align*}
 g_1&=
 \left(\frac1{n-1}+\vartheta_n\right)H^2-|A|^2
 +2C_{\vartheta_n}R^{-2},\\
 g_2&=
 \frac3{n+2}H^2-|A|^2
 +2C_{\vartheta_n}R^{-2}.
\end{align*}
Then \eqref{eq:cylindrical} implies
\begin{equation}\label{eq:g1g2-lower}
 g_1\geq C_{\vartheta_n}R^{-2},
 \qquad
 g_2=g_1+\vartheta_nH^2
 \geq\vartheta_nH^2+C_{\vartheta_n}R^{-2}.
\end{equation}

Equations~\eqref{eq:H2-evol} and \eqref{eq:A2-evol} give the exact
evolutions
\begin{align*}
 (\partial_t-\Delta)g_1
 &={}
 2\left(|\nabla A|^2
 -\left(\frac1{n-1}+\vartheta_n\right)|\nabla H|^2\right)\\
 &\quad
 +2(|A|^2-n\kappa^2)
 \left(g_1-2C_{\vartheta_n}R^{-2}\right)
 -4n\kappa^2\left(|A|^2-\frac1nH^2\right),\\
 (\partial_t-\Delta)g_2
 &={}
 2\left(|\nabla A|^2-\frac3{n+2}|\nabla H|^2\right)\\
 &\quad
 +2(|A|^2-n\kappa^2)
 \left(g_2-2C_{\vartheta_n}R^{-2}\right)
 -4n\kappa^2\left(|A|^2-\frac1nH^2\right).
\end{align*}
Here $|A|^2\geq H^2/n\geq n\kappa^2$ and
$g_i\geq C_{\vartheta_n}R^{-2}$, so
$g_i-2C_{\vartheta_n}R^{-2}\geq-g_i$.  Moreover,
\begin{equation*}
 \frac3{n+2}-\left(\frac1{n-1}+\vartheta_n\right)=\vartheta_n,
\end{equation*}
and the Kato inequality gives
$|\nabla A|^2\geq3|\nabla H|^2/(n+2)$.  We therefore obtain
\begin{align}
 (\partial_t-\Delta)g_1
 &\geq
 \frac{2\vartheta_n(n+2)}3|\nabla A|^2
 -2(|A|^2+n\kappa^2)g_1 \notag\\
 &\quad
 -4n\kappa^2
 \left(|A|^2-\frac1nH^2\right),
 \label{eq:g1-gradient}\\
 (\partial_t-\Delta)g_2
 &\geq
 -2(|A|^2+n\kappa^2)g_2
 -4n\kappa^2
 \left(|A|^2-\frac1nH^2\right).
 \label{eq:g2-gradient}
\end{align}

As in \cite{LN21}*{Theorem~4.9}, consider
\begin{equation*}
 F:=\frac{|\nabla A|^2}{g_1g_2}.
\end{equation*}
At an interior positive spatial maximum of $F$, 
\begin{equation*}
 \frac{\nabla(|\nabla A|^2)}{|\nabla A|^2}
 =\frac{\nabla g_1}{g_1}+\frac{\nabla g_2}{g_2}.
\end{equation*}
Combining \eqref{eq:g1-gradient} - \eqref{eq:g2-gradient} with  \eqref{eq:derivative-evolutions}, and using the quotient rule together with $|\nabla(|\nabla A|^2)|^2
\leq4|\nabla A|^2|\nabla^2A|^2$, we obtain
\begin{align}
 0\leq(\partial_t-\Delta)F
 &\leq
 F\Bigg[
 C(n)(|A|^2+\kappa^2)
 +4n\kappa^2
 \left(|A|^2-\frac1nH^2\right)
 \left(\frac1{g_1}+\frac1{g_2}\right)\notag\\
 &\qquad\quad 
 -\frac{2\vartheta_n(n+2)}3\frac{|\nabla A|^2}{g_1}
 \Bigg].
 \label{eq:gradient-F-maximum}
\end{align}
 The Cauchy inequality and
Proposition~\ref{prop:uniform-two-convexity} give
$0\leq|A|^2-H^2/n\leq|A|^2\leq nH^2$.  Together with
$\kappa^2\leq R^{-2}$ and \eqref{eq:g1g2-lower}, this gives
\begin{align*}
 |A|^2+\kappa^2
 &\leq Cg_2,\\
 \frac{\kappa^2(|A|^2-H^2/n)}{g_1}
 &\leq C\left(|A|^2-\frac1nH^2\right)
 \leq Cg_2,\\
 \frac{\kappa^2(|A|^2-H^2/n)}{g_2}
 &\leq C\kappa^2\leq Cg_2.
\end{align*}
Since $|\nabla A|^2/g_1=Fg_2$,
\eqref{eq:gradient-F-maximum} implies
\begin{equation*}
 0\leq Fg_2
 \left(C-\frac{2\vartheta_n(n+2)}3F\right).
\end{equation*}
Thus every interior positive maximum satisfies $F\leq C(n,\alpha)$.  The
same bound holds at $t=t_0$ by \eqref{eq:short-time-smoothing} with $m=1$
and $g_1g_2\geq C_{\vartheta_n}^2R^{-4}$.  For each $t<T$, the maximum
principle on $M\times[t_0,t]$ therefore gives $F\leq C(n,\alpha)$.  Since
$t<T$ is arbitrary, the bound holds on $M\times[t_0,T)$.

This implies  
\begin{equation*}
 |\nabla A|^2=g_1g_2F\leq C(n,\alpha)(H^4+R^{-4}).
\end{equation*}
Since $H^2\leq n|A|^2$, this proves \eqref{eq:gradient-est}.
\end{proof}

\subsection{Second and higher derivative estimates}

The gradient estimate allows us to control the lower-order terms in the
evolution equation for $\nabla^2A$.  We next derive the corresponding
pointwise estimate for the second derivatives of the second fundamental form.

\begin{theorem}\label{thm:hessian}
  Under the same assumptions on $M_0$, there exists
  $C=C(n,\alpha)<\infty$, independent of $\kappa$, such that
  \begin{equation}\label{eq:hessian-est}
    |\nabla^2A|^2
    \leq C\bigl(|A|^6+R^{-6}\bigr)
  \end{equation}
  on $M_t$ for every $t\in[\tau_nR^2,T)$.
\end{theorem}

\begin{proof}
  This is the weighted maximum-principle argument of
  \cite{HS09}*{Theorem~6.3} and \cite{LN21}*{Theorem~4.12}, with the
  hyperbolic reaction retained below.  The preserved lower bound for the mean
  curvature gives
  \begin{equation}\label{s4.H}
  H\geq n\kappa+\alpha_2R^{-1},
  \qquad
  R^{-1}\leq\alpha_2^{-1}H,
  \qquad
  \kappa\leq n^{-1}H
  \end{equation}
  on $M\times[\tau_nR^2,T)$.  The case $m=2$ of
  \eqref{eq:short-time-smoothing} supplies the initial Hessian bound.
  By Proposition~\ref{prop:uniform-two-convexity} and
  Theorem~\ref{thm:gradient},
  \begin{equation*}
  |A|^2+\kappa^2\leq C(n)H^2,
  \qquad
  |\nabla A|^2\leq C(n,\alpha)H^4,
  \qquad
  |\nabla H|^2\leq C(n,\alpha)H^4.
  \end{equation*}
 Fix
  $C=C(n,\alpha)\geq1$ large enough for all estimates below.

  Using \eqref{eq:derivative-evolutions-2}, the quotient rule, and
  $(\partial_t-\Delta)H\geq0$, we obtain
  \begin{align*}
    (\partial_t-\Delta)\frac{|\nabla^2A|^2}{H^5}
    &\leq
    -2\frac{|\nabla^3A|^2}{H^5}
    +C\frac{|\nabla^2A|^2}{H^3}
    +CH^3\\
    &\quad
    +10H^{-6}
    \langle\nabla(|\nabla^2A|^2),\nabla H\rangle.
  \end{align*}
  The last term is estimated by
  \begin{align*}
    10H^{-6}
    \bigl|\langle\nabla(|\nabla^2A|^2),\nabla H\rangle\bigr|
    &\leq
    20H^{-6}|\nabla^2A||\nabla^3A||\nabla H|\\
    &\leq
    \frac12\frac{|\nabla^3A|^2}{H^5}
    +C\frac{|\nabla^2A|^2}{H^3}.
  \end{align*}
  Hence
  \begin{equation}\label{eq:hessian-weighted-evolution}
    (\partial_t-\Delta)\frac{|\nabla^2A|^2}{H^5}
    \leq
    -\frac{|\nabla^3A|^2}{H^5}
    +C\frac{|\nabla^2A|^2}{H^3}
    +CH^3.
  \end{equation}

  In the same way, using the evolution inequality \eqref{eq:derivative-evolutions} for $|\nabla A|^2$ and the
  gradient estimate above, we find
  \begin{align}\label{eq:gradient-weighted-evolution}
    (\partial_t-\Delta)\frac{|\nabla A|^2}{H^3}
    \leq&
    -2\frac{|\nabla^2A|^2}{H^3}
    +CH^3
    +6H^{-4}
    \langle\nabla(|\nabla A|^2),\nabla H\rangle\nonumber\\
     \leq &
    -\frac32\frac{|\nabla^2A|^2}{H^3}
    +CH^3.
  \end{align}

  The estimate \eqref{s4.H} and $\kappa\leq R^{-1}$ implies
   \begin{equation*}
  n\kappa
  \leq nR^{-1}
  \leq\frac n{\alpha_2}(H-n\kappa)
  \end{equation*}
  and so 
  \begin{equation*}
  H-n\kappa
  \geq\frac{\alpha_2}{\alpha_2+n}H. 
  \end{equation*}
  Then the evolution of $H$ satisfies 
  \begin{align}\label{eq:H-positive-reaction}
    (\partial_t-\Delta)H
    =&~(|A|^2-n\kappa^2)H\nonumber\\
    \geq &~  
    \frac{(H-n\kappa)(H+n\kappa)}{n}\,H  \geq\frac{\alpha_2}{n(\alpha_2+n)}H^3.
  \end{align}

  The quotient $|\nabla^2A|^2/H^5$ has the correct scale, but its evolution \eqref{eq:hessian-weighted-evolution}
  leaves two positive errors of the form
  $C|\nabla^2A|^2/H^3$ and $CH^3$.  We observe that a positive multiple of the evolution \eqref{eq:gradient-weighted-evolution} of  $|\nabla A|^2/H^3$ supplies a negative
  Hessian term that absorbs the first error, while a negative multiple of the evolution \eqref{eq:H-positive-reaction} of $H$
  uses the positive reaction of $H$ to absorb the second. 
  
  To make the cancellation explicit, multiply
  \eqref{eq:gradient-weighted-evolution} by $2(C+1)$, subtract
  $2n(\alpha_2+n)(C+1)^2/\alpha_2$ times
  \eqref{eq:H-positive-reaction}, and add the result to
  \eqref{eq:hessian-weighted-evolution}.  This gives
  \begin{align*}
  (\partial_t-\Delta)\Bigg[
  &\frac{|\nabla^2A|^2}{H^5}
  +2(C+1)\frac{|\nabla A|^2}{H^3} -\frac{2n(\alpha_2+n)(C+1)^2}{\alpha_2}H
  \Bigg]\notag\\
  &\leq
  -\frac{|\nabla^3A|^2}{H^5}
  +\bigl(C-3(C+1)\bigr)\frac{|\nabla^2A|^2}{H^3}\notag\\
  &\qquad
  +\bigl(C+2C(C+1)-2(C+1)^2\bigr)H^3\\
  &\leq
  -\frac{|\nabla^3A|^2}{H^5}
  -\frac{|\nabla^2A|^2}{H^3}
  \leq0.
  \end{align*}
  At $t=\tau_nR^2$, \eqref{eq:short-time-smoothing} with $m=1,2$ and
  $H\geq\alpha_2R^{-1}$ show that the quantity in square brackets is at
  most $C(n,\alpha)R^{-1}$.
  Applying the maximum principle on $M\times[\tau_nR^2,t]$ for each $t<T$
  gives the same bound throughout $M\times[\tau_nR^2,T)$.  Dropping the nonnegative
  middle term in that quantity, we obtain
  \begin{equation*}
  \frac{|\nabla^2A|^2}{H^5}
  \leq
  \frac{2n(\alpha_2+n)(C+1)^2}{\alpha_2}H
  +C(n,\alpha)R^{-1}
  \leq C(n,\alpha)H.
  \end{equation*}
  Thus
  \begin{equation*}
  |\nabla^2A|^2\leq C(n,\alpha)H^6.
  \end{equation*}
  Since $H^2\leq n|A|^2$, this implies
  \eqref{eq:hessian-est}.
\end{proof}

Higher spatial derivatives follow from the standard Bernstein induction, and
mixed time derivatives follow from the evolution equation for the second
fundamental form.  This is the argument of Huisken--Sinestrari
\cite{HS09}*{Theorem~6.3 and Corollary~6.4}.  Compare also
\cite{LN21}*{Theorem~4.15}.  The additional hyperbolic commutator terms are
lower order because the ambient curvature is parallel, and are controlled by
$\kappa^2\leq n^{-2}H^2$.

\begin{corollary}[Higher derivatives]\label{cor:higher-derivatives}
  For every pair of nonnegative integers $a,b$, there exists
  $C_{a,b}=C_{a,b}(n,\alpha)<\infty$ such that
  \begin{equation}\label{eq:higher-derivative-estimate}
    |\nabla_t^a\nabla^bA|^2
    \leq
    C_{a,b}\left(
    H^{4a+2b+2}+R^{-4a-2b-2}
    \right)
  \end{equation}
  on $M_t$ for every $t\in[\tau_nR^2,T)$.  Here $\nabla_t$ denotes the
  covariant time derivative associated with the evolving metric.
\end{corollary}

For a smooth flow, if $p=X(x,t)$ for $x\in M$, let
\begin{equation*}
    p_s:=X(x,s), \qquad s\leq t,
\end{equation*}
be the material track of $p$, and set
\begin{equation*}
 \mathcal P(p,t,\rho,\theta)
 =\bigcup_{s\in[t-\theta,t]}
 B_{g(s)}(p_s,\rho)\times\{s\}.
\end{equation*}
We use this notation only
when $t-\theta\geq0$ and the indicated balls are contained in the smooth
flow.

\begin{corollary}
\label{cor:parabolic-control}
There are constants
\begin{equation*}
 h_\#=h_\#(n,\alpha)<\infty,
 \qquad c_\#=c_\#(n,\alpha)<\infty
\end{equation*}
such that, whenever $H(p,t)\geq h_\#R^{-1}$,
\begin{equation}\label{eq:H-derivative-ratios}
 |\nabla H|(p,t)\leq c_\#H(p,t)^2,
 \qquad
 |\partial_tH|(p,t)\leq\frac12c_\#^2H(p,t)^3.
\end{equation}
If the cylinder
\begin{equation*}
 \mathcal P\left(
 p,t,\frac1{10c_\#H(p,t)},
       \frac1{100c_\#^2H(p,t)^2}\right)
\end{equation*}
is defined, then
\begin{equation}\label{eq:parabolic-H-comparison}
 \frac1{10}H(p,t)\leq H(q,s)\leq10H(p,t)
\end{equation}
throughout this cylinder.
\end{corollary}

\begin{proof}
The ODE estimate used in Lemma~\ref{lem:time-upper} bounds
$H\leq C(n)R^{-1}$ on $[0,\tau_nR^2]$.  Choose the constant $h_\#$ in the
statement so large that this short-time bound is at most
$h_\#R^{-1}/20$.  Wherever $H\geq h_\#R^{-1}/10$, the point lies in the
time range of Theorems~\ref{thm:gradient} and~\ref{thm:hessian}.  Since
\begin{equation*}
 |\nabla H|^2\leq n|\nabla A|^2,
 \qquad
 \partial_tH=\Delta H+(|A|^2-n\kappa^2)H,
\end{equation*}
those estimates give
\begin{equation*}
 |\nabla H|\leq c_\#H^2,
 \qquad
 |\partial_tH|\leq\frac12c_\#^2H^3
\end{equation*}
throughout this larger region, after $c_\#$ is chosen sufficiently large,
depending only on $n,\alpha$.  This proves
\eqref{eq:H-derivative-ratios} at every point in the statement.
Indeed, $|\Delta H|\leq n|\nabla^2A|$, and every power of $R^{-1}$ is
absorbed by $H\geq h_\#R^{-1}/10$.

Throughout the region $H\geq h_\#R^{-1}/10$,
\begin{equation}\label{eq:reciprocal-H-bounds}
 |\nabla(H^{-1})|\leq c_\#,
 \qquad
 |\partial_t(H^{-2})|\leq c_\#^2.
\end{equation}
Fix $(q,s)$ in the indicated parabolic cylinder.  We first integrate the second inequality in
\eqref{eq:reciprocal-H-bounds} backward from $(p,t)$ to $(p_s,s)$
along the material track, and then integrate the first inequality
along a minimizing $g(s)$-geodesic from $p_s$ to $q$.  As long as
$H\geq h_\#R^{-1}/10$ along these paths, the chosen time and space
radii give
\begin{equation*}
    \frac12 H(p,t)\leq H(q,s)\leq2H(p,t).
\end{equation*}
Since $H(p,t)\geq h_\#R^{-1}$, the lower bound is strictly larger than
$h_\#R^{-1}/10$.  A first-failure argument therefore shows that the
required lower-curvature threshold cannot be lost along either path. The backward path cannot reach the time $\tau_nR^2$ either, since the
short-time bound there is at most $h_\#R^{-1}/20$.  Thus the preceding
integrations are valid throughout the cylinder and yield
\eqref{eq:parabolic-H-comparison}. This is the first-failure
argument of \cite{HS09}*{Lemmas~6.6 and~7.2}.  See also
\cite{LN21}*{Lemmas~4.11 and~4.14}.
\end{proof}

\subsection{Neck detection}

We now combine the cylindrical and derivative estimates to identify
high-curvature regions at which the smallest principal curvature is small.
The relevant scale at a point $(p,t)$ is $(n-1)/{H(p,t)}$, the radius of a round cylinder having the same mean curvature.

\begin{definition}[Necks]
  A point $p\in M$ is the center of an $(\eps,k,L)$-neck of radius $r$ if
  there exist an embedding
  \begin{equation*}
  N:\Sph^{n-1}\times(-L,L)\longrightarrow M
  \end{equation*}
  and a point $\omega_0\in\Sph^{n-1}$ such that $N(\omega_0,0)=p$ and,
  after pulling $X\circ N$ back by $\exp_{X(p)}^{-1}$, rescaling lengths by
  $r^{-1}$, and applying a Euclidean rigid motion in
  $T_{X(p)}\Hh^{n+1}_{\kappa}$, the resulting immersion is
  $\eps$-close in $C^k$ to the standard round cylinder
  \begin{equation*}
  \Sph^{n-1}(1)\times(-L,L)\subset\mathbb R^{n+1}.
  \end{equation*}
\end{definition}

\begin{theorem}[Neck detection]\label{thm:neck-detection}
  Fix $\eps>0$, $k\in\mathbb N$, $L\geq10$, and $\Theta>0$.  There exist
  \begin{equation*}
  \eta_{\mathrm{neck}}
  =\eta_{\mathrm{neck}}(n,\alpha,\eps,k,L,\Theta)>0
  \end{equation*}
  and
  \begin{equation*}
  H_{\mathrm{neck}}
  =H_{\mathrm{neck}}(n,\alpha,\eps,k,L,\Theta)<\infty
  \end{equation*}
  with the following property.

  Let $(p,t)$ be a point of a smooth flow satisfying the hypotheses of this
  section, and set
  \begin{equation*}
  r=\frac{n-1}{H(p,t)}.
  \end{equation*}
  Assume that $ \mathcal P\bigl(p,t,2(L+2)r,\Theta r^2\bigr)$  
  is defined.  If
  \begin{equation*}
  H(p,t)\geq H_{\mathrm{neck}}R^{-1},
  \qquad
  \lambda_1(p,t)\leq\eta_{\mathrm{neck}}H(p,t),
  \end{equation*}
  then $p$ is the center of an $(\eps,k,L)$-neck of radius $r$ at time $t$.
  If $\Theta\geq L^2$, then after rescaling by $r^{-1}$ the flow on the
  corresponding backward cylinder over $[-L^2,0]$ is $\eps$-close in
  parabolic $C^k$ norm to a round shrinking cylinder whose radius at time
  zero is one.
\end{theorem}

\begin{proof}
The Euclidean blow-up and continuation argument is
  \cite{HS09}*{Lemma~7.4}.  Its space-form version is
  \cite{LN21}*{Lemma~4.16}.  We record the two points needed for the
  hyperbolic ambient space and for the full buffer in the statement.  Suppose
  that the conclusion fails.  Then there are smooth
  flows with scales $R_j$ and points $(p_j,t_j)$ such that, with
  $r_j=(n-1)/H(p_j,t_j)$,
  \begin{equation*}
  H(p_j,t_j)R_j\geq j,
  \qquad
  \frac{\lambda_1(p_j,t_j)}{H(p_j,t_j)}\leq\frac1j,
  \end{equation*}
  while the conclusion of the theorem fails at $(p_j,t_j)$.

  Pull the ambient metric back by the exponential map at $X(p_j,t_j)$ and
  parabolically rescale by $r_j^{-1}$.  At the base point the rescaled mean
  curvature is $n-1$, and the ratio of the smallest principal curvature to
  the mean curvature tends to zero.  The rescaled ambient sectional curvature
  is $-(\kappa_jr_j)^2$, where
  \begin{equation*}
    0\leq\kappa_jr_j
    \leq\frac{r_j}{R_j}
    =\frac{n-1}{H(p_j,t_j)R_j}
    \longrightarrow0.
  \end{equation*}
  Thus the rescaled ambient metrics converge smoothly on compact sets to the
  Euclidean metric.

  Corollary~\ref{cor:parabolic-control}, Theorems~\ref{thm:gradient}
  and~\ref{thm:hessian}, and Corollary~\ref{cor:higher-derivatives} give
  smooth subsequential convergence on a fixed backward parabolic neighborhood
  to a Euclidean mean curvature flow. Under the rescaling, the lower-order terms in the convexity and cylindrical estimates are multiplied respectively
by $r_j/R_j$ and $(r_j/R_j)^2$, which tend to zero.  Thus, after first
passing to the limit and then letting the parameters in those
estimates tend to zero, we obtain
  \begin{equation*}
    A_\infty\geq0,
    \qquad H_\infty(0,0)=n-1,
    \qquad \lambda_{1,\infty}(0,0)=0,
    \qquad |A_\infty|^2\leq\frac1{n-1}H_\infty^2.
  \end{equation*}

  In the Euclidean limit,
  \begin{equation*}
  (\partial_s-\Delta)h_i{}^j=|A|^2h_i{}^j.
  \end{equation*}
  Apply Hamilton's tensor strong maximum principle \cite{Ham86} from any
  fixed negative time slice in the limiting neighborhood.  The null
  distribution is parallel on a smaller backward neighborhood, so the limit
  splits locally off a line.  Cauchy's inequality for the remaining
  nonnegative principal curvatures gives the reverse inequality
  $|A_\infty|^2\geq H_\infty^2/(n-1)$.  Equality therefore holds, and the
  remaining principal curvatures are equal.  The Codazzi equations identify
  the limit as a round shrinking cylinder.  Its radius at time zero is one
  because $H_\infty(0,0)=n-1$.

  It remains only to extend the convergence to the full neighborhood in
the statement.  On every compact subset of the limiting shrinking
cylinder, the mean curvature is bounded above and below by positive
constants.  Smooth convergence and
Corollary~\ref{cor:parabolic-control} therefore give the same curvature
comparison on a slightly larger normalized parabolic region.  The
compactness argument may then be repeated there.  Each repetition
enlarges the available space-time buffer by a definite amount depending
only on $n$ and $\alpha$.  Since $L$ and $\Theta$ are fixed, finitely
many repetitions, followed by a diagonal subsequence, give smooth
convergence on $\mathcal P\bigl(0,0,2(L+2),\Theta\bigr)$.  

  The time-zero slices now converge in $C^k$ on the larger buffer region to
  the unit round cylinder.  The standard integration of curvature necks and
  the normal-neck parametrization
  \cite{HS09}*{Propositions~3.4, 3.5 and Theorems~3.12, 3.14}
  then show that, for all sufficiently large $j$, $p_j$ is the center of an
  $(\eps,k,L)$-neck of radius $r_j$.  This contradicts the choice of the
  sequence.  If $\Theta\geq L^2$, the same smooth convergence on the
  backward interval $[-L^2,0]$ gives the asserted parabolic $C^k$-closeness.
\end{proof}

\section{Key estimates for surgically modified flows}
\label{sec:standard-surgery}

From this point onward the ambient sectional curvature is $-1$, and we write
\begin{equation*}
\CC_\tau(R,\alpha):=\CC_{\tau,1}(R,\alpha).
\end{equation*}
Hyperbolic space has infinite injectivity radius, constant curvature, and
parallel curvature tensor.  In particular, the normal-coordinate estimates
used in the surgery construction are uniform in the center of the surgery
neck.

We use the standard replacement of Huisken--Sinestrari
\cite{HS09}*{Section~3}.  An $(\eps,k,L)$-neck of radius $r$ is pulled back
to the tangent space by the exponential map, rescaled by $r^{-1}$,
modified by the standard bending and cap construction, and then mapped
back to $\Hh^{n+1}$.  At this scale the ambient metric is an $O(r^2)$
perturbation of the Euclidean metric.  All necks used for surgery are
understood in the canonical normal parametrization of
\cite{HS09}*{Section~3}.

The surgery-time arguments for the cylindrical, derivative, and
neck-detection estimates are those of \cite{HS09}*{Sections~5--7}.  Their
space-form version appears in \cite{LN21}*{Sections~5.2--5.5}.  We cite these
standard arguments below and give details where the hyperbolic ambient metric
or the quantitative $\tau$-bi-Ricci condition requires an additional check.

\begin{definition}
  A \emph{surgically modified mean curvature flow} is a finite or locally
  finite sequence
  \begin{equation*}
  X^i:M_i\times[T_i,T_{i+1})\longrightarrow\Hh^{n+1},
  \qquad
  0=T_0<T_1<T_2<\cdots,
  \end{equation*}
  of smooth mean curvature flows such that, at every surgery time $T_{i+1}$,
  finitely many disjoint normal necks of a common radius are replaced by
  standard caps.  Components recognized as diffeomorphic to $\Sph^n$ or
  $\Sph^{n-1}\times\Sph^1$ may subsequently be discarded.
\end{definition}

\subsection{The standard replacement and the $\tau$-bi-Ricci condition}

We first verify directly that the quantitative $\tau$-bi-Ricci condition
survives the standard bending.  The calculation follows the same general
principle as the surgery calculation in
\cite{BH17}*{Lemma~7.3}: the axial principal curvature receives a definite
positive increment, whereas the remaining principal curvatures change only
by lower-order terms.

\begin{lemma}[Preservation of the $\tau$-bi-Ricci pinching under surgery]
  \label{lem:surgery-taubiric-preservation}
  Fix $\tau\in\mathcal I_n$ and an initial quantitative pinching constant
  $\alpha_0^{\mathrm{in}}>0$.  There exist surgery parameters
  $B>1$ and $0<\delta_0<1$, and constants
  \begin{equation*}
  \eps_b>0,\qquad
  k_b\geq2,\qquad
  L_b\geq40,\qquad
  r_b>0,
  \end{equation*}
  together with
  \begin{equation*}
  0<\alpha_0^*\leq\alpha_0^{\mathrm{in}},
  \qquad
  c_b>0,
  \end{equation*}
  depending only on $n$, $\tau$, $\alpha_0^{\mathrm{in}}$, and the fixed
  standard cap, with the following property.

  Suppose that a standard replacement is performed on an
  $(\eps,k,L)$-neck of radius $r$ satisfying
  \begin{equation*}
  \eps\leq\eps_b,\qquad
  k\geq k_b,\qquad
  L\geq L_b,\qquad
  r\leq r_b,
  \end{equation*}
  and suppose that on the incoming neck the reduced pinching inequality
  \begin{equation}\label{eq:incoming-surgery-taubiric}
    P_{ij}^{(\tau)}
    \geq
    \alpha_0^*H|\lambda_p-\lambda_q|
  \end{equation}
  for every $i\ne j$ and all $p,q$.  Then every changed point after surgery
  satisfies
  \begin{equation}\label{eq:surgery-taubiric-ratio}
    \widetilde P_{ij}^{(\tau)}
    \geq
    \alpha_0^*\widetilde H
    |\widetilde\lambda_p-\widetilde\lambda_q|
  \end{equation}
  for every $i\ne j$ and all $p,q$.  Moreover,
  \begin{equation}\label{eq:surgery-taubiric-strict}
    \widetilde P_{ij}^{(\tau)}
    \geq c_b r^{-2},
    \qquad
    \widetilde H\geq c_b r^{-1}
  \end{equation}
  throughout the changed region.
\end{lemma}

\begin{proof}
 The proof is divided into three steps.  We first treat the bent portion
of the neck using the relative bending estimates.  We then obtain a
uniform quantitative margin on the normalized transition and cap.
Finally, we absorb the ambient-curvature and normal-coordinate errors
and derive the strict bounds.  The parameters are chosen successively,
beginning with $\theta$ and the fixed bending parameters, followed by
$\alpha_0^*$, the neck quality, and finally the surgery radius.

  It is useful to separate the homogeneous quadratic part of the curvature
  expression from its fixed ambient term.  Since the ambient sectional
  curvature is $-1$, \eqref{eq:Ptau-expanded} may be written as
  \begin{equation}\label{eq:surgery-Phi-definition}
    P_{ij}^{(\tau)}
    =
    \Phi_{ij}^{(\tau)}(\lambda)-c_{n,\tau},
  \end{equation}
  where
  \begin{equation*}
  \Phi_{ij}^{(\tau)}(\lambda)
  :=
  (2-\tau)\lambda_i\lambda_j
  +(\lambda_i+\lambda_j)
  \sum_{\ell\ne i,j}\lambda_\ell,
  \qquad
  c_{n,\tau}=2(n-1)-\tau.
  \end{equation*}

\textbf{Step 1. The bent portion of the neck.}   We first consider the bent portion of the neck.  In the pulled-back normal
  coordinates used for the standard replacement, write
  \begin{equation*}
  u(z)
  =
  r\exp\left(-\frac{B}{z-10}\right),
  \qquad
  z\in[10,40].
  \end{equation*}
  The bent hypersurface in these coordinates is
  \begin{equation*}
  \widetilde N(\omega,z)
  =
  N(\omega,z)-\delta_0u(z)\nu(\omega,z).
  \end{equation*}
  Here $\delta_0$ is the bending parameter in the surgery construction and is
  unrelated to the parameter $\tau$ in the $\tau$-bi-Ricci curvature.

  Let $\lambda_1,\ldots,\lambda_n$ be the principal curvatures at a point of
  the incoming neck, with $\lambda_1$ corresponding to the axial direction,
  and let
  $\widetilde\lambda_1,\ldots,\widetilde\lambda_n$ denote the principal
  curvatures at the corresponding point of the bent hypersurface.  Fix
  $\theta>0$ sufficiently small, depending only on $n$ and $\tau$, so that
  all the absorptions below hold.  The standard construction allows us to
  choose $\delta_0$ sufficiently small and $B$ sufficiently large, and then to
  improve the neck quality and decrease the surgery radius, so that the
  bending estimates give
  \begin{align}
    \left|
    \widetilde\lambda_1-
    \left(
    \lambda_1+\delta_0D_1D_1u+\delta_0u\lambda_1^2
    \right)
    \right|
    &\leq
    \theta\delta_0D_1D_1u,
    \label{eq:surgery-lambda1-raw}\\
    \left|
    \widetilde\lambda_i-
    \left(
    \lambda_i+\delta_0u\lambda_i^2
    \right)
    \right|
    &\leq
    \theta\delta_0D_1D_1u,
    \qquad 2\leq i\leq n.
    \label{eq:surgery-lambdai-raw}
  \end{align}
  In the Euclidean model these are the principal-curvature estimates of
  \cite{HS09}*{Theorem~3.19 and Remark~3.20}.  After re-embedding in
  hyperbolic normal coordinates, the relative Riemannian calculation in
  \cite{BH17}*{the proof of Lemma~7.3} gives the same estimates, with the
  error bounded by $\theta\delta_0D_1D_1u$, once the neck quality is improved
  and $r$ is sufficiently small.

  Increasing $B$ further, we may arrange
  \begin{equation}\label{eq:surgery-bending-smallness}
    u\leq\theta r^2D_1D_1u,
    \qquad
    0\leq D_1D_1u\leq\theta r^{-1}.
  \end{equation}
  These inequalities are the scale-invariant consequences of
  \cite{HS09}*{Lemma~3.18}. Compare again the parameter choice in
  \cite{BH17}*{the proof of Lemma~7.3}.
  Since on an $\eps$-neck one has $|\lambda_i|\leq Cr^{-1}$, the terms
  $u\lambda_i^2$ in
  \eqref{eq:surgery-lambda1-raw}--\eqref{eq:surgery-lambdai-raw}
  are absorbed by the error.  After changing $\theta$ by a fixed
  dimensional factor, we therefore obtain
  \begin{align}
    \left|
    \widetilde\lambda_1-
    \left(\lambda_1+\delta_0D_1D_1u\right)
    \right|
    &\leq
    \theta\delta_0D_1D_1u,
    \label{eq:surgery-lambda1}\\
    |\widetilde\lambda_i-\lambda_i|
    &\leq
    \theta\delta_0D_1D_1u,
    \qquad 2\leq i\leq n.
    \label{eq:surgery-lambdai}
  \end{align}
  Set $d:=\delta_0D_1D_1u$.  
  Then
  \begin{equation}\label{eq:surgery-H-expansion}
    \widetilde H
    =
    H+d+O(\theta d).
  \end{equation}
  Thus, by the choice of $\theta$,
  \begin{equation}\label{eq:surgery-H-bent}
    \widetilde H\geq H+\frac12d\geq H.
  \end{equation}

We now compute the change of $P_{ij}^{(\tau)}$.  The function
$\Phi_{ij}^{(\tau)}$ is linear in each principal curvature separately.
Using \eqref{eq:surgery-lambda1}--\eqref{eq:surgery-lambdai},
$|\lambda_i|\leq Cr^{-1}$, and
$d\leq C\theta r^{-1}$, all terms containing at least one error factor
are absorbed into $O(\theta r^{-1}d)$. Then for
  $2\leq i<j\leq n$,
  \begin{equation}\label{eq:surgery-Pij-spherical}
    \widetilde P_{ij}^{(\tau)}
    =
    P_{ij}^{(\tau)}
    +d(\lambda_i+\lambda_j)
    +O(\theta r^{-1}d),
  \end{equation}
  whereas, for $j\geq2$,
  \begin{equation}\label{eq:surgery-P1j}
    \widetilde P_{1j}^{(\tau)}
    =
    P_{1j}^{(\tau)}
    +d\left(
    (2-\tau)\lambda_j
    +\sum_{\ell\ne1,j}\lambda_\ell
    \right)
    +O(\theta r^{-1}d).
  \end{equation}
  The constant ambient term $-c_{n,\tau}$ cancels in these differences.

  On the cylindrical model with curvature vector
  $(0,r^{-1},\ldots,r^{-1})$,
  \begin{equation*}
  \lambda_1=0,
  \qquad
  \lambda_2=\cdots=\lambda_n=r^{-1}.
  \end{equation*}
  Hence the coefficient of $d$ in
  \eqref{eq:surgery-Pij-spherical} is $2r^{-1}$, while the coefficient in
  \eqref{eq:surgery-P1j} is $(n-\tau)r^{-1}$.
  By decreasing $\eps_b$ while keeping the preceding choice of $\theta$, the
  incoming neck is close enough to the round cylinder that
  \begin{equation}\label{eq:surgery-P-gain}
    \widetilde P_{ij}^{(\tau)}
    \geq
    P_{ij}^{(\tau)}
    +\frac{\min\{2,n-\tau\}}4r^{-1}d
  \end{equation}
  throughout the bent region.

  We next control the right-hand side of the quantitative pinching
  inequality.  Since
  \begin{equation*}
  H\simeq r^{-1},
  \qquad
  |\lambda_i|\leq Cr^{-1},
  \end{equation*}
  equations
  \eqref{eq:surgery-lambda1}--\eqref{eq:surgery-H-expansion} give, for
  arbitrary $p,q$,
  \begin{equation}\label{eq:surgery-width-change}
    \widetilde H
    |\widetilde\lambda_p-\widetilde\lambda_q|
    \leq
    H|\lambda_p-\lambda_q|
    +C(n)r^{-1}d.
  \end{equation}
  Combining
  \eqref{eq:incoming-surgery-taubiric},
  \eqref{eq:surgery-P-gain}, and
  \eqref{eq:surgery-width-change}, we obtain
  \begin{align}
    &\widetilde P_{ij}^{(\tau)}
    -\alpha_0^*\widetilde H
    |\widetilde\lambda_p-\widetilde\lambda_q|
    \notag\\
    &\qquad\geq
    \left(
    \frac{\min\{2,n-\tau\}}4-C(n)\alpha_0^*
    \right)r^{-1}d.
    \label{eq:surgery-quantitative-margin}
  \end{align}
  We may therefore choose
  \begin{equation}\label{eq:surgery-alpha-star-choice}
    \alpha_0^*
    \leq
    \min\left\{
    \frac12\alpha_0^{\mathrm{in}},
    \frac{\min\{2,n-\tau\}}{8C(n)}
    \right\}.
  \end{equation}
  Then \eqref{eq:surgery-quantitative-margin} is nonnegative, proving
  \eqref{eq:surgery-taubiric-ratio} throughout the bent portion.

  \smallskip
\textbf{Step 2. The transition and cap regions.}
We next consider the transition and cap regions. We are free to decrease $\alpha_0^*$ further at this stage. Since all restrictions imposed on $\alpha_0^*$ in the bent-region argument are upper bounds, the preceding proof of \eqref{eq:surgery-taubiric-ratio} remains valid with the reduced value of $\alpha_0^*$.  

For a principal-curvature vector $\lambda$ at scale $r$, set
$\mu_i=r\lambda_i$. 
  Then
  \begin{equation}\label{eq:normalized-surgery-P}
    r^2P_{ij}^{(\tau)}
    =
    \Phi_{ij}^{(\tau)}(\mu)-c_{n,\tau}r^2.
  \end{equation}
  On the unit cylinder $(0,1,\ldots,1)$,
  \begin{equation}\label{eq:surgery-cylinder-Phi}
    \Phi_{1j}^{(\tau)}=n-2,
    \qquad
    \Phi_{ij}^{(\tau)}
    =2(n-2)-\tau,
    \quad 2\leq i<j\leq n.
  \end{equation}
  Within the analytically allowed range $0\leq\tau\leq2$, both quantities are
  strictly positive precisely for $\tau\in\mathcal I_n$.

By \cite{HS09}*{Corollary~3.21 and Theorem~3.22}, after the parameters
$B$ and $\delta_0$ have been fixed, the normalized transition and cap
are described by a fixed smooth model.  On the portion close to the
bent cylinder, \eqref{eq:surgery-cylinder-Phi} gives a uniform positive
lower bound for every $\Phi_{ij}^{(\tau)}$.  On the remaining fixed
strictly convex portion, the same conclusion follows from
$0\leq\tau\leq2$.  Compactness of the normalized model therefore gives
a uniform positive lower bound for all $\Phi_{ij}^{(\tau)}$, and this
bound persists for all sufficiently good incoming necks.

The normalized mean curvature and all normalized spectral widths are
uniformly bounded on the same region.  We may therefore decrease
$\alpha_0^*$, still subject to \eqref{eq:surgery-alpha-star-choice},
and then choose $c_b>0$ so that
\begin{equation}\label{eq:normalized-cap-taubiric}
 \Phi_{ij}^{(\tau)}(\mu)
 \geq
 2\alpha_0^*
 \left(\sum_\ell\mu_\ell\right)
 |\mu_p-\mu_q|+2c_b
\end{equation}
for every $i\ne j$ and all $p,q$ throughout the normalized transition
and cap.

\textbf{Step 3. Hyperbolic errors and the strict bounds.}   It remains to compare the normalized Euclidean model with the hyperbolic
  metric.  In geodesic normal coordinates at the center of the neck, let
  $D_r(y)=ry$.
  On every fixed ball containing the normalized replacement,
  \begin{equation}\label{eq:hyperbolic-normal-metric}
    \left\|
    r^{-2}D_r^*\bigl(\exp_{X(p)}^*g_{\Hh}\bigr)
    -g_{\mathbb R^{n+1}}
    \right\|_{C^m}
    \leq C(n,m)r^2
  \end{equation}
  for every fixed $m$.  The normalized replacement immersions have uniform
  $C^{m+2}$ bounds on this fixed region.  The second fundamental form depends
  smoothly on the $C^2$ jets of the ambient metric and the immersion.  Hence
  \eqref{eq:hyperbolic-normal-metric} gives $O(r^2)$ control of the normalized
  Weingarten maps and therefore of their eigenvalues, mean curvatures, and the
  two sides of \eqref{eq:normalized-cap-taubiric}.  In view of
  \eqref{eq:normalized-surgery-P}, the strict margin in
  \eqref{eq:normalized-cap-taubiric} absorbs both $c_{n,\tau}r^2$ and the
  normal-coordinate error when $r_b$ is sufficiently small.  This proves
  \eqref{eq:surgery-taubiric-ratio} on the entire changed region.

 The same estimates also give the uniform strict bounds in
\eqref{eq:surgery-taubiric-strict}.  On the bent portion, the incoming
neck is uniformly close after rescaling to the round cylinder, and
\eqref{eq:surgery-lambda1}--\eqref{eq:surgery-lambdai} preserve this
closeness after improving the neck quality.  Hence
\eqref{eq:surgery-cylinder-Phi}, together with a further decrease of
$r_b$, gives a fixed positive normalized lower bound for
$\widetilde P_{ij}^{(\tau)}$.  On the transition and cap the same
conclusion follows from the strict margin in
\eqref{eq:normalized-cap-taubiric}.

Moreover, on the bent portion
\eqref{eq:surgery-H-bent} and the neck geometry give
$\widetilde H\geq H\simeq r^{-1}$, while the normalized transition and
cap have mean curvature bounded uniformly away from zero.  Reducing
$c_b$ and $r_b$ if necessary, we therefore obtain
\begin{equation*}
    \widetilde P_{ij}^{(\tau)}
    \geq c_b r^{-2},
    \qquad
    \widetilde H\geq c_b r^{-1}
\end{equation*}
throughout the changed region.  This proves
\eqref{eq:surgery-taubiric-strict}.
\end{proof}

We now collect the remaining standard estimates for the replacement.
Unlike the preceding lemma, these estimates require no new
$\tau$-bi-Ricci algebra.

\begin{lemma}[Estimates for the standard replacement]
  \label{lem:standard-replacement}
  Fix $0<\eta<1$ and an integer $m_0\geq2$.  After fixing
  $\alpha_0^*$ and the surgery parameters in
  Lemma~\ref{lem:surgery-taubiric-preservation}, the neck quality may be
  improved and the surgery radius decreased, without changing
  $\alpha_0^*$, so that there exist
  \begin{equation*}
  \eps_s>0,\qquad
  k_s\geq m_0+2,\qquad
  L_s\geq L_b,\qquad
  r_s>0,
  \end{equation*}
  and constants $c_s>0$, $C_m<\infty$, $0\leq m\leq m_0$, with the
  following property.

  If surgery is performed on an $(\eps,k,L)$-neck of radius $r$ satisfying
  \begin{equation*}
  \eps\leq\eps_s,\qquad
  k\geq k_s,\qquad
  L\geq L_s,\qquad
  r\leq r_s,
  \end{equation*}
  and if \eqref{eq:incoming-surgery-taubiric} holds on the incoming neck,
  then every changed point satisfies
  \begin{align}
    \widetilde P_{ij}^{(\tau)}
    &\geq
    \alpha_0^*\widetilde H
    |\widetilde\lambda_p-\widetilde\lambda_q|,
    \label{eq:standard-replacement-taubiric}\\
    \widetilde P_{ij}^{(\tau)}
    &\geq c_sr^{-2},
    \qquad
    \widetilde H\geq c_sr^{-1},
    \label{eq:standard-replacement-strict}\\
    |\widetilde A|^2-
    \left(\frac1{n-1}+\eta\right)\widetilde H^2
    &\leq
    -c_s\widetilde H^2,
    \label{eq:surgery-cylindrical-margin}\\
    r^{m+1}|\widetilde\nabla^m\widetilde A|
    &\leq C_m,
    \qquad
    0\leq m\leq m_0.
    \label{eq:surgery-derivative-model}
  \end{align}
  The area removed from each changed neck is at least $c_sr^n$, and the
  replacement may be chosen so that
  \begin{equation}\label{eq:surgery-minimum-H}
    \min_{\widetilde M}\widetilde H
    \geq
    \min_MH.
  \end{equation}

  If $0<R\leq1$ and
  \begin{equation}\label{eq:surgery-radius-class-bound}
    r\leq\frac{c_sR}{n+\alpha_2},
  \end{equation}
  then the replacement preserves
  $\CC_\tau(R,(\alpha_0^*,\alpha_1,\alpha_2))$.
\end{lemma}

\begin{proof}
  The estimates \eqref{eq:standard-replacement-taubiric} and
\eqref{eq:standard-replacement-strict} follow from
  Lemma~\ref{lem:surgery-taubiric-preservation}, after decreasing $c_s$ if
  necessary.

  The cylindrical, derivative, and area estimates for the Euclidean
  replacement are standard.  We use the construction of
  \cite{HS09}*{Section~3, in particular Corollary~3.21 and Theorem~3.22},
  together with \cite{HS09}*{Theorem~5.3(ii)} applied with parameter
  $\eta/2$.  On the fixed normalized replacement these conclusions have
  strict margins.  Applying \eqref{eq:hyperbolic-normal-metric} with
  $m=m_0+2$ makes the hyperbolic curvature quantities $C^{m_0}$-close to
  their Euclidean counterparts.
  Decreasing $r_s$ therefore preserves the strict cylindrical margin and the
  derivative bounds.  It also changes the Euclidean area difference only by
  $O(r^{n+2})$, so the area loss remains at least $c_sr^n$.

  For the non-strict comparison of the minimum of $H$, an absolute
  $O(r^2)$ perturbation estimate is not sufficient near the unchanged collar,
  where the Euclidean gain tends to zero.  On the bent portion the relative
  estimate gives
  \begin{equation*}
  \widetilde H-H
  \geq
  \frac12\delta_0D_1D_1u\geq0.
  \end{equation*}
  At the boundary of the changed region the deformation and all its
  derivatives vanish, so the two hypersurfaces agree exactly.  On the part a
  fixed normalized distance from this boundary, the Euclidean interpolation
  and cap have a fixed positive mean-curvature margin over the unit cylinder
  by \cite{HS09}*{Corollary~3.21 and Theorem~3.22}.  The incoming neck satisfies
  $rH=(n-1)+O(\eps)$.  After first decreasing $\eps_s$ and then $r_s$, the
  positive model margin absorbs both the neck error and the hyperbolic
  normal-coordinate error.  Thus $\widetilde H\geq H$ there, while the relative
  estimate handles the remaining collar.  Consequently
  \eqref{eq:surgery-minimum-H} holds.  Compare also
  \cite{BH17}*{the proof of Lemma~7.3}.

  Finally, on a changed region
  \begin{equation*}
  \widetilde H\geq c_sr^{-1}.
  \end{equation*}
  Under \eqref{eq:surgery-radius-class-bound},
  \begin{equation*}
  \widetilde H
  \geq
  (n+\alpha_2)R^{-1}
  \geq
  n+\alpha_2R^{-1},
  \end{equation*}
  where $R\leq1$ was used in the last inequality.  Thus the lower
  mean-curvature condition is preserved.  The area does not increase,
  the quantitative $\tau$-bi-Ricci condition is preserved with the fixed
  constant $\alpha_0^*$, and discarded components only remove points.
  Therefore the replacement preserves
  $\CC_\tau(R,(\alpha_0^*,\alpha_1,\alpha_2))$.
\end{proof}

\begin{corollary}[Preservation of the surgery class]
  \label{cor:surgery-class-preserved}
  Fix the parameters so that every standard replacement satisfies
  Lemma~\ref{lem:standard-replacement}, including
  \eqref{eq:surgery-radius-class-bound}.   Reduce the quantitative pinching constant obtained from
  Theorem~\ref{thm:taubiric-preservation} once to the constant
  $\alpha_0^*$ in Lemma~\ref{lem:surgery-taubiric-preservation}.  The original
  inequality implies the reduced one because $\alpha_0^*$ was chosen no
  larger than the original constant.  Theorem~\ref{thm:taubiric-preservation}
  preserves it between surgery times, and
  Lemma~\ref{lem:surgery-taubiric-preservation} preserves the same constant at
  every standard replacement.  Hence the entire surgically modified flow
  preserves
  $\CC_\tau(R,(\alpha_0^*,\alpha_1,\alpha_2))$.  In particular, the pinching
  constant does not deteriorate with the number of surgeries.
\end{corollary}

\begin{lemma}[Total elapsed time]\label{lem:surgery-time-upper}
  For every surgically modified flow in the fixed surgery class,
  \begin{equation}\label{eq:surgery-total-time}
    \sum_i(T_{i+1}-T_i)
    \leq
    \frac n2\alpha_2^{-2}R^2.
  \end{equation}
\end{lemma}

\begin{proof}
  On every smooth interval, $m(t):=\min_{M_t}H$ satisfies  
  \begin{equation*}
  m'   \geq 
  \frac1n m(m^2-n^2)
  \end{equation*}
 in the barrier sense.  At a surgery time, \eqref{eq:surgery-minimum-H} shows that $m$ has no
  downward jump, while discarding components can only increase the minimum.
  Since $m(0)\geq n+\alpha_2R^{-1}$,  
  integration of the scalar inequality gives
  \begin{align*}
    \sum_i(T_{i+1}-T_i)
    &\leq
    \frac1{2n}
    \log
    \frac{(n+\alpha_2R^{-1})^2}
    {\alpha_2R^{-1}(2n+\alpha_2R^{-1})}\\
    &\leq
    \frac n2\alpha_2^{-2}R^2.
  \end{align*}
  For a locally finite sequence the same estimate applies to every finite
  prefix, and the conclusion follows by taking the supremum.
\end{proof}

\subsection{Convexity and cylindrical estimates}

\begin{theorem}[Convexity and cylindrical estimates through surgery]
  \label{thm:surgical-cylindrical-estimates}
  Assume that
  \begin{equation*}
  M_0\in\CC_\tau(R,\alpha),
  \qquad
  |A|^2\leq R^{-2}\quad\text{on }M_0.
  \end{equation*}
 For every $\varepsilon,\eta>0$, the dimensionless surgery parameters and the
  upper bound for the ratio of the surgery radius to $R$ may be chosen,
  depending only on   $n$, $\tau$, $\alpha$, $\varepsilon$, and $\eta$, so that every   surgically modified flow satisfies
  \begin{align}
    \lambda_1
    &\geq
    -\varepsilon H-C_\varepsilon R^{-1},
    \label{eq:surgical-convexity}\\
    |A|^2-\frac1{n-1}H^2
    &\leq
    \eta H^2+C_\eta R^{-2}.
    \label{eq:surgical-cylindrical}
  \end{align}
  Here
  \begin{equation*}
  C_\varepsilon
  =C_\varepsilon(n,\tau,\alpha,\varepsilon),
  \qquad
  C_\eta
  =C_\eta(n,\tau,\alpha,\eta),
  \end{equation*}
  and neither constant depends on the number or the times of the surgeries.
\end{theorem}

\begin{proof}
  It suffices to consider $0<\varepsilon,\eta<1/2$.  Choose the standard
  replacement so that Lemma~\ref{lem:standard-replacement} holds with
  parameter $\eta/2$.

  Use the function $f_{\sigma,\eta/2}$ from
  Section~\ref{sec:cylindrical}.  At a surgery time its positive part vanishes
  on every changed region by \eqref{eq:surgery-cylindrical-margin}, agrees
  with its pre-surgery value on the unchanged collars, and is removed on
  discarded components.  Hence $\int (f_{\sigma,\eta/2})_+^p\,d\mu$  
  has no upward jump at surgery.

  The differential inequalities of Section~\ref{sec:cylindrical} hold on
  every smooth interval.  The preceding jump inequality allows their
  standard integration through surgery times exactly as in
  \cite{HS09}*{Theorem~5.3} and \cite{LN21}*{Theorem~5.3}.  The area bound and
  Lemma~\ref{lem:surgery-time-upper} supply the two global bounds used there.
  Thus the same $L^p$ estimates and Stampacchia iteration give
  \begin{equation*}
  |A|^2-\frac1{n-1}H^2
  \leq
  \eta H^2+C_\eta R^{-2},
  \end{equation*}
  with a constant independent of the number of surgeries.  This proves
  \eqref{eq:surgical-cylindrical}.

  To obtain convexity, apply \eqref{eq:surgical-cylindrical} with parameter
  $n\varepsilon^2/(n-1)$.  The pointwise argument in the proof of
  Corollary~\ref{cor:convexity} gives \eqref{eq:surgical-convexity}, after
  increasing $C_\varepsilon$.
  Choosing the replacement parameters so that
  Lemma~\ref{lem:standard-replacement} applies simultaneously with
  parameters $\eta/2$ and $n\varepsilon^2/(2(n-1))$ proves
  \eqref{eq:surgical-convexity}.
\end{proof}

\subsection{Derivative estimates}

\begin{theorem}[Derivative estimates through surgery]
  \label{thm:estimates-through-surgery}
  Fix an integer $m_0\geq2$ and choose the standard replacement so that
  Lemma~\ref{lem:standard-replacement} holds with parameter $\vartheta_n/2$
  and derivative order $m_0$.  Assume that the first surgery time is not
  earlier than $\tau_nR^2$.  Then there exists $C=C(n,\tau,\alpha,m_0)<\infty$  
  such that every surgically modified flow satisfying the hypotheses of
  Theorem~\ref{thm:surgical-cylindrical-estimates} obeys
  \begin{align}
    |\nabla A|^2
    &\leq
    C(H^4+R^{-4}),
    \label{eq:surgical-gradient}\\
    |\nabla^2A|^2
    &\leq
    C(H^6+R^{-6})
    \label{eq:surgical-hessian}
  \end{align}
  for every $t\geq\tau_nR^2$.

  More generally, for every pair of nonnegative integers $a,b$ with
  $2a+b\leq m_0$,
  \begin{equation}\label{eq:surgical-higher-derivatives}
    |\nabla_t^a\nabla^bA|^2
    \leq
    C_{a,b}
    \left(
    H^{4a+2b+2}+R^{-4a-2b-2}
    \right),
  \end{equation}
  where $ C_{a,b}
  =  C_{a,b}(n,\tau,\alpha,m_0,a,b)$.  
  At a surgery time the derivatives are understood one-sidedly.
\end{theorem}

\begin{proof}
  On every smooth interval these are the estimates of
  Section~\ref{sec:gradient-estimates}.  It remains only to verify that the
  maximum-principle barriers can be restarted after surgery.

  On every changed region,
  \eqref{eq:standard-replacement-strict} and
  \eqref{eq:surgery-derivative-model} give
  \begin{equation*}
  H\simeq r^{-1},
  \qquad
  |\nabla^mA|
  \leq
  C_mr^{-m-1},
  \qquad
  0\leq m\leq m_0.
  \end{equation*}
  In particular,
  \begin{equation*}
  |\nabla A|^2\leq CH^4,
  \qquad
  |\nabla^2A|^2\leq CH^6.
  \end{equation*}
  The strict cylindrical margin makes both denominators in the quotient used
  in Theorem~\ref{thm:gradient} bounded below by a fixed positive multiple of
  $H^2$.  The quotient is therefore uniformly bounded on every inserted cap.

  For the Hessian estimate the same model bounds give
  \begin{equation*}
  \frac{|\nabla^2A|^2}{H^5}
  +C\frac{|\nabla A|^2}{H^3}
  \leq CH.
  \end{equation*}
  The coefficient of the negative $H$ term in the barrier from
  Theorem~\ref{thm:hessian} may be enlarged once, without changing its
  evolution inequality, so that this barrier is nonpositive on every changed
  region.  There is no jump on unchanged collars, while discarded components
  only remove points.  The maximum-principle arguments of
  \cite{HS09}*{Theorem~6.3} and \cite{LN21}*{Theorems~5.4--5.5} therefore
  restart after every surgery with constants independent of the number of
  surgeries.

  Higher spatial and mixed time derivatives follow from the standard
  induction in \cite{HS09}*{Theorem~6.3 and Corollary~6.4}.  Since the
  hyperbolic
  curvature tensor is parallel, the ambient terms are lower order, and the
  cap bounds \eqref{eq:surgery-derivative-model} provide the required
  post-surgery bounds up to order $m_0$.  This proves
  \eqref{eq:surgical-higher-derivatives}.
\end{proof}

\subsection{Neck detection}

For a surgically modified flow, a subset of a later time-slice is called
\emph{trackable} over a time interval if its material parametrization
passes only through unchanged regions and is identified canonically across
every intervening surgery time.  A backward parabolic neighborhood is
called \emph{surgery-free} if all of its spatial slices are trackable and
their tracks are disjoint from every changed or discarded region.  On a
surgery-free neighborhood the pieces of the flow paste together to form a
single smooth mean curvature flow.

\begin{theorem}[Neck detection for surgically modified flows]
  \label{thm:surgical-neck-detection}
  Fix $\eps>0$, $k\in\mathbb N$, $L\geq10$, and $\Theta>0$, and choose
the standard replacement with $m_0\geq k+2$.  There exist
  \begin{equation*}
  \eta_{\mathrm{surg}}
  =
  \eta_{\mathrm{surg}}(n,\tau,\alpha,\eps,k,L,\Theta)>0
  \end{equation*}
  and
  \begin{equation*}
  H_{\mathrm{surg}}
  =
  H_{\mathrm{surg}}(n,\tau,\alpha,\eps,k,L,\Theta)<\infty
  \end{equation*}
  with the following property.

  Let $(p,t)$ lie in a surgically modified flow satisfying
  Theorem~\ref{thm:estimates-through-surgery}, assume
  $t\geq\tau_nR^2$, and put $r=(n-1)/H(p,t)$.  
  Suppose that
  \begin{equation*}
  H(p,t)\geq H_{\mathrm{surg}}R^{-1},
  \qquad
  \lambda_1(p,t)\leq\eta_{\mathrm{surg}}H(p,t),
  \end{equation*}
  and that $\mathcal P\bigl(p,t,2(L+2)r,\Theta r^2\bigr)$  
  is surgery-free.  Then $p$ is the center of an $(\eps,k,L)$-neck of
  radius $r$.  If $\Theta\geq L^2$, the rescaled flow is $\eps$-close in
  parabolic $C^k$ norm on the corresponding backward cylinder to a round
  shrinking cylinder whose radius at time zero is one.
\end{theorem}

\begin{proof}
  On the surgery-free neighborhood the flow is smooth.  The convexity and
  cylindrical estimates required in the proof of
  Theorem~\ref{thm:neck-detection} are supplied by
  Theorem~\ref{thm:surgical-cylindrical-estimates}, while the curvature and
  derivative bounds are supplied by
  Theorem~\ref{thm:estimates-through-surgery}.  The contradiction and
  blow-up argument of Theorem~\ref{thm:neck-detection} therefore applies on
  this neighborhood.  This is the surgery-free alternative in
  \cite{HS09}*{Lemma~7.4} and \cite{LN21}*{Theorem~5.6}.  Under the rescaling,
  the ambient sectional curvature tends to zero, so the limit is the same
  round Euclidean shrinking cylinder as in the smooth proof.  The assumption
  $m_0\geq k+2$ supplies the required $C^k$ compactness.
\end{proof}

\section{Existence of a terminating surgically modified flow}
\label{sec:parameter-choice}

We now complete the construction of the flow.  We say that a surgically
modified flow \emph{terminates} if, at a smooth time or immediately after
the final standard replacements, every remaining connected component is
recognized as diffeomorphic to $\Sph^n$ or
$\Sph^{n-1}\times\Sph^1$ and is discarded.

The construction is the surgery algorithm of Huisken--Sinestrari
\cite{HS09}*{Sections~3, 7, and~8}.  We choose its parameters in the same
order as in \cite{HS09}*{pp.~208--210}.  Fix first a finite derivative order
$m_0\geq2$ large enough for all neck-detection tests and for the normal-neck
and cap construction.  Let
$\alpha_0^*$ be the permanent pinching constant supplied by
Lemma~\ref{lem:surgery-taubiric-preservation} for the incoming constant
$\alpha_0$.
Since $\alpha_0^*\leq\alpha_0$, the initial hypersurface also belongs to the
reduced class $\CC_\tau(R,(\alpha_0^*,\alpha_1,\alpha_2))$, which is preserved by
Corollary~\ref{cor:surgery-class-preserved}.  We then fix the finite
collection of accuracy parameters needed in the neck-continuation argument
and refine the neck and surgery parameters so that all conclusions of
Section~\ref{sec:standard-surgery} hold simultaneously.

Choose curvature thresholds
\begin{equation*}
H_3>H_2>H_1\gg R^{-1}
\end{equation*}
in the order prescribed by the surgery construction, and put $r_*:={(n-1)}/{H_1}$. 
Increase $H_1$ so that $r_*$ is smaller than every surgery-radius
threshold in Section~\ref{sec:standard-surgery}, including the
threshold required to preserve the reduced surgery class, and also
smaller than the normal-coordinate threshold needed in the Riemannian
neck-continuation argument.  The
ratios $H_2/H_1$ and $H_3/H_2$ are then chosen sufficiently large for the
neck-selection procedure, and so that every new standard cap has mean
curvature strictly below $H_2$.  We also choose $H_3$ larger than the
short-time upper curvature bound on $[0,\tau_nR^2]$, so that the first
surgery occurs after $\tau_nR^2$.

\begin{proposition}[Standard neck continuation and selection]
\label{prop:standard-neck-continuation}
For the parameters fixed above, the neck-continuation and neck-selection
construction of \cite{HS09}*{Lemmas~7.4, 7.10, and~7.12,
Theorem~8.2, and the proof of Theorem~1.1 on pp.~216--218} applies to the
surgically modified flow.  Whenever $H_{\max}$ first reaches $H_3$, it produces
finitely many pairwise disjoint normal necks of radius $r_*$.  Standard
replacement on these necks, followed by discarding the recognized
high-curvature components, gives
\begin{equation}\label{eq:Hmax-after-surgery}
  H_{\max}(t+)\leq H_2.
\end{equation}
All constants are uniform over the surgery times.
\end{proposition}

\begin{proof}
The required uniform two-convexity follows from
Proposition~\ref{prop:uniform-two-convexity} and
Corollary~\ref{cor:surgery-class-preserved}.  The convexity and cylindrical
estimates are given by Theorem~\ref{thm:surgical-cylindrical-estimates}, and
the derivative estimates are given by
Theorem~\ref{thm:estimates-through-surgery}.  If a backward continuation stays
disjoint from all earlier changed regions, Theorem~\ref{thm:surgical-neck-detection}
gives the surgery-free alternative of \cite{HS09}*{Lemma~7.4}.  If it meets a
previous replacement, the fixed cap geometry and
\eqref{eq:standard-replacement-strict}--\eqref{eq:surgery-derivative-model}
give the surgery properties \textup{(s1)--(s3)} on pp.~199--200 of \cite{HS09}.  The previous-cap
alternative then follows from \cite{HS09}*{Lemmas~7.10 and~7.12}.

The integration and continuation of normal necks are given in
\cite{HS09}*{Proposition~7.18 and Theorem~8.2}.  Only the local
ambient-metric changes from \cite{BH17}*{Section~8} are needed here.  Compare
also the spherical mean curvature flow argument in
\cite{LN21}*{Section~6}.  After rescaling a surgery neck of radius $r_*$ to
unit size, the ambient sectional curvature is $-r_*^2$, and
\eqref{eq:hyperbolic-normal-metric} gives a metric error of order $r_*^2$.
All accuracy parameters have already been fixed, so increasing $H_1$ absorbs
these errors.  The neck selection in
\cite{HS09}*{the proof of Theorem~1.1 on pp.~216--218} gives the stated
conclusion.
\end{proof}

\begin{theorem}[Existence and termination]
  \label{thm:terminating-surgery-flow}
  Let $M_0\in\CC_\tau(R,\alpha)$ satisfy $ |A|^2\leq R^{-2}$ on $M_0$.  
  For the fixed parameters chosen above, there exists a surgically modified
  mean curvature flow starting from $M_0$ which terminates after finitely
  many surgery times.  Moreover,
  \begin{equation}\label{eq:terminating-flow-time}
    \sum_i(T_{i+1}-T_i)
    \leq
    \frac n2\alpha_2^{-2}R^2.
  \end{equation}
\end{theorem}

\begin{proof}
  The inductive construction is the one in \cite{HS09}*{Section~8}, with the
  Riemannian modifications described above.  Whenever $H_{\max}$ reaches
  $H_3$, Proposition~\ref{prop:standard-neck-continuation} gives the required
  disjoint surgery necks and
  \eqref{eq:Hmax-after-surgery} allows the smooth flow to restart.
  Corollary~\ref{cor:surgery-class-preserved} and
  Theorems~\ref{thm:surgical-cylindrical-estimates}
  and~\ref{thm:estimates-through-surgery} keep all constants uniform.

  Each standard replacement is performed at the fixed radius $r_*$ and
  decreases the area by at least $c_sr_*^n$.  Smooth mean curvature flow and
  discarding components do not increase area.  Since $|M_0|\leq\alpha_1R^n$,  
  the total number of neck replacements is finite.  Cutting one neck increases
  the number of connected components by at most one.  Since the initial
  hypersurface is connected, the total number of components created during
  the construction is therefore at most one plus the total number of neck
  replacements.  A trigger event with no replacement discards at least one
  component.  Hence only finitely many discard-only events and only finitely
  many surgery times can occur.

  Suppose that a nonterminal component survived after the last event.  While
  $H_{\max}<H_3$, Proposition~\ref{prop:uniform-two-convexity} gives
  $|A|^2\leq nH_3^2$, so the standard continuation criterion extends the
  smooth flow. On the other hand, the minimum mean curvature satisfies the same differential inequality used in Lemma~\ref{lem:surgery-time-upper},
\begin{equation*}
    m'\geq\frac1n m(m^2-n^2),
\end{equation*}
and therefore becomes unbounded in finite additional time.
Consequently $H_{\max}$ must reach $H_3$ before that time.
Proposition~\ref{prop:standard-neck-continuation} would then produce
another replacement or discard event, contradicting the choice of the
last event.  Hence the flow terminates.  Finally,
Lemma~\ref{lem:surgery-time-upper} gives
\eqref{eq:terminating-flow-time}.
\end{proof}

\section{Topology and proof of the main theorem}
\label{sec:topology}

We now recover the topology of the initial hypersurface from the terminating
surgery flow and complete the proof of Theorem~\ref{thm:taubiric-main}.   Suppose first that $n=3$ and $\tau=2$.  In dimension three, positive
  $2$-bi-Ricci curvature is exactly positive Ricci curvature.  By
  Lemma~\ref{lem:taubiric-mean-convex}, there is a global normal with $H>3$.
  In a principal frame,
  \begin{equation*}
  \Ric(e_i,e_i)=\lambda_i(H-\lambda_i)-2>0.
  \end{equation*}
  If $\lambda_i\leq0$, then $H-\lambda_i>0$ and the right-hand side is at most
  $-2$, a contradiction.  Thus all principal curvatures are positive.  The
  hyperbolic Hadamard--Stoker theorem \cite{Alexander77} shows that $X_0$ is an
  embedded sphere bounding a convex ball.  The smooth convergence theorem of
  Andrews--Chen \cite{AC17}*{Theorem~2} then shows that the ordinary mean
  curvature flow converges to a round point.  This proves the theorem at the
  endpoint with no surgery times.

  We may now assume that $\tau\in\mathcal I_n$.  By
  Lemma~\ref{lem:taubiric-mean-convex}, after choosing the orientation we have
  $H>n$.  Since $M_0$ is compact and has strictly positive $\tau$-bi-Ricci
  curvature, there exists $\alpha_0>0$ such that
  \begin{equation*}
  P_{ij}^{(\tau)}
  \geq
  \alpha_0H|\lambda_p-\lambda_q|
  \end{equation*}
  for every $i\ne j$ and all $p,q$.  Choose $R\leq 1$ so that $ |A|^2\leq R^{-2}$  
  on $M_0$, and then choose $\alpha_1,\alpha_2>0$ such that
  \begin{equation*}
  |M_0|\leq\alpha_1R^n,
  \qquad
  H-n\geq\alpha_2R^{-1}.
  \end{equation*}
  Thus
  \begin{equation*}
  M_0\in\CC_\tau(R,\alpha),
  \qquad
  \alpha=(\alpha_0,\alpha_1,\alpha_2).
  \end{equation*}

  Theorem~\ref{thm:taubiric-preservation} preserves the quantitative
  $\tau$-bi-Ricci condition on every smooth time interval.  The one-time
  reduction $\alpha_0\mapsto\alpha_0^*$ and the preservation of the reduced
  surgery class are built into Sections~\ref{sec:standard-surgery}
  and~\ref{sec:parameter-choice}.
  Theorem~\ref{thm:terminating-surgery-flow}
  therefore produces a surgically modified mean curvature flow starting from
  $M_0$ and terminating after finitely many surgery times.

  It remains to recover the topology.  A standard neck replacement cuts along
  an embedded copy of $\Sph^{n-1}$ and caps the resulting boundary components
  by $n$-balls.  The topological effect of reversing such a replacement is
  described in \cite{HS09}*{Proposition~3.23}: either two connected components
  are joined by a connected sum, or a single component acquires a
  $\Sph^{n-1}\times\Sph^1$ summand.

  The terminal alternatives in the surgery algorithm are a spherical or
  capped-neck component, a closed normal neck diffeomorphic to
  $\Sph^{n-1}\times\Sph^1$, or a strictly convex component.  In the last
  case, the hyperbolic Hadamard--Stoker theorem
  \cite{Alexander77} applies and shows that a closed locally strictly convex
  hypersurface immersion in hyperbolic space is an embedded sphere bounding a
  convex ball.  Since the surgery process contains only finitely many
  replacements, reversing them one at a time and using that $M$ is connected
  gives
  \begin{equation*}
  M\cong\Sph^n
  \quad\text{or}\quad
  M\cong
  \mathop{\#}_{j=1}^{k}
  \bigl(\Sph^{n-1}\times\Sph^1\bigr)
  \end{equation*}
  for some $k\geq1$.

  Suppose now that $X_0$ is embedded and bounds a compact domain $\Omega_0$.
  The standard replacement can be chosen to preserve embeddedness.  The
  Euclidean normal-neck result is
  \cite{HS09}*{Proposition~3.25 and Theorem~3.26}.  At a surgery neck,
  \eqref{eq:hyperbolic-normal-metric} shows that the rescaled hyperbolic metric
  is $C^{k_s}$-close to the Euclidean metric with error $O(r_*^2)$.  For a
  sufficiently large surgery curvature scale, the replacement therefore lies
  in the same embedded normal tube and both caps lie on the mean-convex side.
  This is the Riemannian modification used in \cite{BH17}*{Section~7}.  Thus
  every surgery slice remains embedded and bounds a compact domain.

  At the terminal time, the remaining domains form a finite disjoint union of
  balls and solid tubes.  Each solid tube is a ball with one one-handle.
  Reversing a neck replacement either attaches a one-handle to one component
  or joins two components by a boundary connected sum.  Reversing all
  surgeries and using the connectedness of $\Omega_0$ therefore reconstructs
  $\Omega_0$ as a ball with finitely many one-handles.   This completes the proof.


\section*{Acknowledgements}
This work was supported by the National Key Research and Development Program of China (2021YFA1001800), National Natural Science Foundation of China (No.~12531002) and the Fundamental Research Funds for the Central Universities.  The third author was also supported by the China Postdoctoral Science Foundation (Grant No.~2025M783146). The authors used ChatGPT (OpenAI) during the preparation of this 
manuscript for language editing and for comments on the clarity,
exposition, and organization of the presentation. All suggestions were
independently assessed, and all mathematical arguments, statements,
calculations and references were independently verified by the authors,
who take full responsibility for the content of the manuscript.

\end{document}